\documentclass[12pt]{amsart}
\usepackage{amsthm}
\usepackage{amssymb}
\usepackage{amsmath}
\usepackage{color}
\theoremstyle{plain}
\newtheorem{proposition}{Proposition}
\newtheorem{theorem}[proposition]{Theorem}

\newtheorem{corollary}[proposition]{Corollary}

\theoremstyle{definition}
\newtheorem{definition}[proposition]{Definition}

\theoremstyle{definition}

\newtheorem{remark}[proposition]{Remark}

\numberwithin{equation}{section}

\numberwithin{proposition}{section}

\gdef\myletter{}

\let\savetheequation\theequation
\def\theequation{\savetheequation\myletter}

\usepackage{color}

\newcommand{\CC}{{\mathbb C}}

\newcommand{\RR}{{\mathbb R}}
\newcommand{\ZZ}{{\mathbb Z}}

\newcommand{\PP}{{\mathbb P}}

\newcommand{\NN}{{\mathbb N}}

\renewcommand{\date}{\today}
\def \bar{\overline}
\def \hat{\widehat}

\begin{document}


\title[$C-$Robin]{\bf $C-$Robin Functions and Applications}

\author{Norm Levenberg* and Sione Ma`u}{\thanks{*Supported by Simons Foundation grant No. 354549}}
\subjclass{32U15, \ 32U20, \ 32U40}%
\keywords{convex body, $C-$Robin function}%


\address{Indiana University, Bloomington, IN 47405 USA}
\email{nlevenbe@indiana.edu}

\address{University of Auckland, Auckland, New Zealand}  
\email{s.mau@auckland.ac.nz}


\maketitle

\begin{abstract} We continue the study in \cite{BBL} in the setting of pluripotential theory arising from polynomials associated to a convex body $C$ in $(\RR^+)^d$. Here we discuss $C-$Robin functions and their applications. In the particular case where $C$ is a simplex in $(\RR^+)^2$ with vertices $(0,0),(b,0),(a,0)$, $a,b>0$, we generalize results of T. Bloom to construct families of polynomials which recover the $C-$extremal function $V_{C,K}$ of a nonpluripolar compact set $K\subset \CC^d$.  \end{abstract}

\setcounter{tocdepth}{2}
\tableofcontents
\vfill \eject

\section{Introduction}\label{sec:intro} As in \cite{BBL}, we fix a convex body $C\subset (\RR^+)^d$ and we define the logarithmic indicator function
\begin{equation}\label{logind}H_C(z):=\sup_{J\in C} \log |z^J|:=\sup_{(j_1,...,j_d)\in C} \log[|z_1|^{j_1}\cdots |z_d|^{j_d}].\end{equation}
We assume throughout that 
\begin{equation}\label{sigmainkp} \Sigma \subset kC \ \hbox{for some} \ k\in \ZZ^+  \end{equation}
where 
$$\Sigma:=\{(x_1,...,x_d)\in \RR^d: 0\leq x_i \leq 1, \ \sum_{j=1}^d x_i \leq 1\}.$$ 
Then
$$H_C(z)\geq \frac{1}{k}\max_{j=1,...,d}\log^+ |z_j| =\frac{1}{k} H_{\Sigma}(z)$$
where $\log^+ |z_j| =\max[0,\log|z_j|]$. We define
$$L_C=L_C(\CC^d):= \{u\in PSH(\CC^d): u(z)- H_C(z) =O(1), \ |z| \to \infty \},$$ and 
$$L_{C,+}=L_{C,+}(\CC^d)=\{u\in L_C(\CC^d): u(z)\geq H_C(z) + C_u\}$$
where $PSH(\CC^d)$ denotes the class of plurisubharmonic functions on $\CC^d$. These are generalizations of the classical Lelong classes $L:=L_{\Sigma}, \ L^+:=L_{\Sigma,+}$ when $C=\Sigma$. Let $\CC[z]$ denote the polynomials in $z$ and 
\begin{equation}\label{polync} Poly(nC):=\{p\in\CC[z]\colon  p(z)=\sum_{\alpha\in nC} a_{\alpha}z^{\alpha} \}.\end{equation}
For a nonconstant polynomial $p$ we define 
\begin{equation}\label{cdegree}\deg_C(p)=\min\{ n\in\NN\colon p\in Poly(nC)\}.\end{equation} 
If $p\in Poly(nC), \ n\geq 1$ we have $\frac{1}{n}\log |p|\in L_C$; also each $u\in L_{C,+}(\CC^d)$ is locally bounded in $\CC^d$. 
For $C=\Sigma$, we write $Poly(nC)=\mathcal P_n$. 

The $C$-extremal function of a compact set $K\subset\CC^2$ is defined as the uppersemicontinuous (usc) regularization $V_{C,K}^*(z):=\limsup_{\zeta \to z} V_{C,K}(\zeta)$ of 
$$
V_{C,K}(z):=\sup\{u(z)\colon u\in L_C, u\leq 0 \hbox{ on } K\}.
$$
If $K$ is regular ($V_K:=V_{\Sigma,K}$ is continuous), then $V_{C,K}=V_{C,K}^*$ is continuous (cf., \cite{BosLev}). In particular, for $K=T^d=\{(z_1,...,z_d)\in \CC^d: |z_j|=1, \ j=1,...,d\}$, $V_{C,T^d}=V_{C,T^d}^*=H_C$ (cf., (2.7) in \cite{BBL}). If $K$ is not pluripolar, i.e., for any $u$ psh with $u=-\infty$ on $K$ we have $u\equiv -\infty$, the Monge-Amp\`ere measure $(dd^cV_{C,K}^*)^d$ is a positive measure with support in $K$ and $V_{C,K}^*=0$ quasi-everywhere (q.e.) on supp$(dd^cV_{C,K}^*)^d$ (i.e., everywhere except perhaps a pluripolar set).

Much of the recent development of this $C-$pluripotential theory can be found in \cite{BosLev}, \cite{BBL} and \cite{BBLL}. One noticeable item lacking from these works is a constructive approach to finding natural concrete families of polynomials associated to $K,C$ which recover $V_{C,K}$. In order to do this, following the approach of Tom Bloom in  \cite{Bapp} and \cite {Bfam}, we introduce a $C-$Robin function $\rho_u$ for a function $u\in L_C$. The ``usual'' Robin function ${\bf \rho_u}$ associated to $u\in L_{\Sigma}$ is defined as 
\begin{equation}\label{stdrobin} {\bf \rho_u}(z):=\limsup_{|\lambda|\to \infty} [u(\lambda z)-\log |\lambda|] \end{equation}
and this detects the asymptotic behavior of $u$. This definition is natural since the ``growth function'' $H_{\Sigma}(z)=\max_{j=1,...,d}\log^+ |z_j| $ satisfies $H_{\Sigma}(\lambda z) = H_{\Sigma}(z)+ \log |\lambda|$. Let 
$C$ be the triangle in $\RR^2$ with vertices $(0,0), (b,0), (0,a)$ where $a,b$ are relatively prime positive integers. Then 
\begin{enumerate}
\item $H_C(z_1,z_2)=\max[\log^+|z_1|^b, \log^+|z_2|^a]$ (note $H_C=0$ on the closure of the unit polydisk $P^2:=\{(z_1,z_2): |z_1|,|z_2|< 1\}$);
\item defining $\lambda \circ (z_1,z_2):=(\lambda^az_1,\lambda^b z_2)$, we have 
$$H_C(\lambda \circ (z_1,z_2))=H_C(z_1,z_2)+ab\log|\lambda|$$ for $(z_1,z_2)\in \CC^2 \setminus P^2$ and $|\lambda| \geq 1$.
\end{enumerate}
Given $u\in L_C(\CC^2)$, we define the $C-$Robin function of $u$ (Definition \ref{crobtri}) as 
$$\rho_u(z_1,z_2):=\limsup_{|\lambda|\to \infty} [u(\lambda \circ (z_1,z_2))-ab\log|\lambda|]$$
for $(z_1,z_2)\in \CC^2$. This agrees with (\ref{stdrobin}) when $a=b=1$; i.e., when $C=\Sigma \subset (\RR^+)^2$. For general convex bodies $C$, it is unclear how to define an analogue to recover the asymptotic behavior of $u\in L_C$. 

The next two sections give some general results in $C-$pluripotential theory which will be used further on but are of independent interest. Section 4 begins in earnest with the case where $C$ is a triangle in $\CC^2$. The key results utilized in our analysis are the use of an integral formula of Bedford and Taylor \cite{BT}, Theorem \ref{bt5.5} in section 6, yielding the fundamental Corollary \ref{keyresult}, and recent results on $C-$transfinite diameter in \cite{SDRNA} and \cite{Sione} of the second author in section 7. Our arguments in Sections 5 and 8 follow closely those of Bloom in \cite{Bapp} and \cite {Bfam}. The main theorem, Theorem \ref{mainthm}, is stated and proved in section 8; then explicit examples of families of polynomials which recover $V_{C,K}$ are provided. We mention that the results given here for triangles $C$ in $\RR^2$ with vertices $(0,0), (b,0), (0,a)$ where $a,b$ are relatively prime positive integers should generalize to the case of a simplex 
$$C=co\{(0,...,0),(a_1,0,...,0),...,(0,...,0,a_d)\}$$
in $(\RR^+)^d, \ d>2$ where $a_1,...,a_d$ are pairwise relatively prime (cf., Remark \ref{generald}). Section 9 indicates generalizations to weighted situations. 

\section{Rumely formula and transfinite diameter}

We recall the definition of $C-$transfinite diameter $\delta_C(K)$ of a compact set $K\subset \CC^d$ where $C$ satisfies (\ref{sigmainkp}). Letting $N_n$ be the dimension of $Poly(nC)$ in (\ref{polync}), we have
$$Poly(nC)= \hbox{span} \{e_1,...,e_{N_n}\}$$ 
where $\{e_j(z):=z^{\alpha(j)}=z_1^{\alpha_1(j)} \cdots z_d^{\alpha_d(j)}\}_{j=1,...,N_n}$ are the standard basis monomials in $Poly(nC)$ in any order. For 
points $\zeta_1,...,\zeta_{N_n}\in \CC^d$, let
$$VDM(\zeta_1,...,\zeta_{N_n}):=\det [e_i(\zeta_j)]_{i,j=1,...,N_n}  $$
$$= \det
\left[
\begin{array}{ccccc}
 e_1(\zeta_1) &e_1(\zeta_2) &\ldots  &e_1(\zeta_{N_n})\\
  \vdots  & \vdots & \ddots  & \vdots \\
e_{N_n}(\zeta_1) &e_{N_n}(\zeta_2) &\ldots  &e_{N_n}(\zeta_{N_n})
\end{array}
\right]$$
and for a compact subset $K\subset \CC^d$ let
$$V_n =V_n(K):=\max_{\zeta_1,...,\zeta_{N_n}\in K}|VDM(\zeta_1,...,\zeta_{N_n})|.$$
Then
$$\delta_C(K):= \limsup_{n\to \infty}V_{n}^{1/l_n}$$
is the $C-$transfinite diameter of $K$ where $l_n:=\sum_{j=1}^{N_n} {\rm deg}_C(e_j)$. The existence of the limit is not obvious but in this setting it is proved in \cite{BBL}. We return to this issue in section 7.

Next, for $u,v \in L_{C,+}$, we define the mutual energy 
\begin{equation}
\label{relendef}
\mathcal E (u,v):= \int_{\CC^d} (u-v)\sum_{j=0}^d (dd^cu)^j\wedge (dd^cv)^{d-j}.
\end{equation}
Here $dd^c=i\partial \bar \partial$ and for locally bounded psh functions, e.g., for $u,v \in L_{C,+}$, the complex Monge-Amp\`ere operators $(dd^cu)^j\wedge (dd^cv)^{d-j}$ are well-defined as positive measures. We have that $\mathcal E$ satisfies the cocycle property; i.e., for $u,v,w\in L_{C,+}$, (cf., \cite{BBL}, Proposition 3.3) 
$${\mathcal E}(u,v) +{\mathcal E}(v,w) + {\mathcal E}(w,u)=0.$$

Connecting these notions, we recall the following formula from \cite{BBL}.
\begin{theorem} \label{energyrumely} Let $K\subset \CC^d$ be compact and nonpluripolar. Then
$$
\log \delta_C(K)=  \frac{-1}{c}\mathcal E(V_{C,K}^*,H_C)
$$
where $c$ is a positive constant depending only on $d$ and $C$.
\end{theorem}

We will use the global domination principle for general $L_C$ and $L_{C,+}$ classes associated to convex bodies satisfying (\ref{sigmainkp}) (cf., \cite{LP}):

\begin{proposition} \label{gctp} For $C\subset (\RR^+)^d$ satisfying (\ref{sigmainkp}), let $u \in L_C$ and $v \in L_{C,+}$ with $u \leq v$ a.e.-$(dd^cv)^d$. Then $u \leq v$ in $\CC^d$.
\end{proposition}

We use these ingredients to prove the following.

\begin{proposition} \label{tdv} Let $E\subset F$ be compact and nonpluripolar. If $\delta_C(E)=\delta_C(F)$ then 
$V_{C,E}^*=V_{C,F}^*$. 
\end{proposition}

\begin{proof} By Theorem \ref{energyrumely}, the hypothesis implies that $\mathcal E(V_{C,E}^*,H_C)=\mathcal E(V_{C,F}^*,H_C)$. Using the cocycle property,
$$0=\mathcal E(V_{C,E}^*,H_C) +\mathcal E(H_C,V_{C,F}^*) +\mathcal E(V_{C,F}^*,V_{C,E}^*)$$
$$=\mathcal E(V_{C,E}^*,H_C) -\mathcal E(V_{C,F}^*,H_C) +\mathcal E(V_{C,F}^*,V_{C,E}^*)$$
$$=\mathcal E(V_{C,F}^*,V_{C,E}^*).$$
From the definition (\ref{relendef}),
$$0=\mathcal E(V_{C,F}^*,V_{C,E}^*)=\int_{\CC^d} (V_{C,F}^*-V_{C,E}^*)\sum_{j=0}^d (dd^cV_{C,F}^*)^j\wedge (dd^cV_{C,E}^*)^{d-j}$$
$$=\int_{\CC^d} (V_{C,F}^*-V_{C,E}^*) (dd^cV_{C,F}^*)^d+\int_{\CC^d} (V_{C,F}^*-V_{C,E}^*) (dd^cV_{C,E}^*)^d$$
$$+\int_{\CC^d} (V_{C,F}^*-V_{C,E}^*)\sum_{j=1}^{d-1} (dd^cV_{C,F}^*)^j\wedge (dd^cV_{C,E}^*)^{d-j}.$$

Now $E\subset F$ implies $V_{C,F}^*\leq V_{C,E}^*$; i.e., $V_{C,F}^*- V_{C,E}^*\leq 0 \ \hbox{on} \ \CC^d$. Also, 
$$V_{C,F}^*=V_{C,E}^* = 0 \ \hbox{q.e. on supp}(dd^cV_{C,E}^*)^d$$
and
$$V_{C,F}^* = 0 \ \hbox{q.e. on supp}(dd^cV_{C,F}^*)^d.$$
Thus we see that
$$0=\int_{\CC^d} (-V_{C,E}^*) (dd^cV_{C,F}^*)^d+\int_{\CC^d} (V_{C,F}^*-V_{C,E}^*)\sum_{j=1}^{d-1} (dd^cV_{C,F}^*)^j\wedge (dd^cV_{C,E}^*)^{d-j}$$
where each term on the right-hand-side is nonpositive. Hence each term vanishes. In particular,
$$0=\int_{\CC^d} V_{C,E}^* (dd^cV_{C,F}^*)^d$$
implies that $V_{C,E}^*=0$ q.e. on supp$(dd^cV_{C,F}^*)^d$ (and hence a.e.-$(dd^cV_{C,F}^*)^d$). 

We finish the proof by using the domination principle (Proposition \ref{gctp}): we have $V_{C,E}^*,V_{C,F}^*\in L_{C,+}(\CC^d)$ with 
$$V_{C,E}^*=0\leq V_{C,F}^* \ \hbox{a.e.}-(dd^cV_{C,F}^*)^d$$
and hence $V_{C,E}^*\leq V_{C,F}^*$ on $\CC^d$; i.e., $V_{C,E}^* =V_{C,F}^*$ on $\CC^d$.

\end{proof}

\begin{remark} For $C=\Sigma$, Proposition \ref{tdv} was proved for regular compact sets $E,F$ in \cite{Bfam} and in general (compact and nonpluripolar sets) in \cite{BloomLev}. Both results utilized the ``usual'' Robin functions (\ref{stdrobin}) of $V_E^*,V_F^*$.

\end{remark}

\section{Other preliminary results: General} Let $K\subset \CC^d$ be compact and nonpluripolar and let $\mu$ be a positive measure on $K$ such that one can form orthonormal polynomials $\{p_{\alpha}\}$ using Gram-Schmidt on the monomials $\{z^{\alpha}\}$. 
We use the notion of degree given in (\ref{cdegree}): $\deg_C(p)=\min\{ n\in\NN\colon p\in Poly(nC)\}$. We have the Siciak-Zaharjuta type polynomial formula 
\begin{equation}\label{siczah} V_{C,K}(z)=\sup\{\tfrac{1}{\deg_C(p)}\log|p(z)|\colon p\in\CC[z], \|p\|_K\leq 1\}\end{equation}
(cf., \cite{BBL}, Proposition 2.3). It follows that $\{z\in \CC^d: V_{C,K}(z)=0\}=\hat K$, the polynomial hull of $K$:
$$\hat K:=\{z\in \CC^d: |p(z)|\leq ||p||_K, \ \hbox{all polynomials} \ p\}.$$

In this section, we follow the arguments of Zeriahi in \cite{Zer}.

\begin{proposition} \label{threeone} In this setting,
$$\limsup_{|\alpha|\to \infty} \tfrac{1}{\deg_C(p_{\alpha})}\log|p_{\alpha}(z)|\geq V_{C,K}(z), \ z\not\in \hat K.$$
\end{proposition}  

\begin{proof} Let $Q_n:=\sum_{\alpha \in nC}c_{\alpha}p_{\alpha}\in Poly(nC)$ with $||Q_n||_K\leq 1$. Then
$$|c_{\alpha}|=|\int_K Q_n\bar p_{\alpha} d\mu| \leq \int_K |\bar p_{\alpha}| d\mu\leq \sqrt{\mu(K)}$$
by Cauchy-Schwarz. Hence 
$$|Q_n(z)|\leq N_n \sqrt{\mu(K)} \max_{\alpha \in nC}|p_{\alpha}(z)|$$ where recall $N_n=$dim$(Poly(nC))$.

Now fix $z_0\in \CC^d\setminus \hat K$ and let $\alpha_n\in nC$ be the multiindex with $\deg_C(p_{\alpha_n})$ largest such that 
$$|p_{\alpha_n}(z_0)|=\max_{\alpha \in nC}|p_{\alpha}(z_0)|.$$
We claim that taking any sequence $\{Q_n\}$ with $||Q_n||_K\leq 1$ for all $n$, 
$$\lim_{n\to \infty} \deg_C(p_{\alpha_n})=+\infty.$$
For if not, then by the above argument, there exists $A< +\infty$ such that for any $n$ and any $Q_n\in Poly(nC)$ with $||Q_n||_K\leq 1$,
$$|Q_n(z_0)|\leq N_n \sqrt{\mu(K)} \max_{\deg_C(p_{\alpha})\leq A }|p_{\alpha}(z_0)|=N_n M(z_0)$$
where $M(z_0)$ is independent of $n$. But then
$$V_{C,K}(z_0)=\sup\{\tfrac{1}{\deg_C(p)}\log|p(z)|\colon p\in\CC[z], \|p\|_K\leq 1\}$$
$$\leq \limsup_{n\to \infty} [\frac{1}{n}\log N_n +\frac{1}{n}\log M(z_0)]=0$$
which contradicts $z_0 \in \CC^d\setminus \hat K$. We conclude that for any $z \in \CC^d\setminus \hat K$, for any $n$ and any $Q_n\in Poly(nC)$ with $||Q_n||_K\leq 1$,
$$\frac{1}{n}\log|Q_n(z)|\leq \frac{1}{n}\log N_n +\frac{1}{n} \log |p_{\alpha_n}(z)|$$
where we can assume $\deg_C(p_{\alpha_n})\uparrow +\infty$. Hence, for such $z$, 
$$V_{C,K}(z)\leq \limsup_{n \to \infty} \tfrac{1}{n}\log|p_{\alpha_n}(z)| \leq \limsup_{n \to \infty} \tfrac{1}{\deg_C(p_{\alpha_n})}\log|p_{\alpha_n}(z)|$$
$$\leq \limsup_{|\alpha|\to \infty} \tfrac{1}{\deg_C(p_{\alpha})}\log|p_{\alpha}(z)|$$
where we have used  $\deg_C(p_{\alpha_n})\leq n$.

\end{proof}

Suppose $\mu$ is any {\it Bernstein-Markov measure} for $K$; i.e., for any $\epsilon >0$, there exists a constant $c_{\epsilon}$ so that
$$||p_n||_K\leq c_{\epsilon}(1+\epsilon)^n ||p_n||_{L^2(\mu)}, \ p_n\in Poly(nC), \ n=1,2,...$$
From (\ref{sigmainkp}), $\Sigma \subset kC \subset m\Sigma$ for some $k,m$ and we can replace $(1+\epsilon)^n$ by $(1+\epsilon)^{\deg_C(p_n)}$. In particular, for the orthonormal polynomials $\{p_{\alpha}\}$,
$$||p_{\alpha}||_K \leq c_{\epsilon}(1+\epsilon)^{\deg_C(p_{\alpha})}.$$ 
Thus
$$\limsup_{|\alpha|\to \infty} \tfrac{1}{\deg_C(p_{\alpha})}\log||p_{\alpha}||_K\leq 0$$
and we obtain equality in the previous result:

\begin{corollary}\label{bmext} In this setting, if $\mu$ is any Bernstein-Markov measure for $K$,
$$\limsup_{|\alpha|\to \infty} \tfrac{1}{\deg_C(p_{\alpha})}\log|p_{\alpha}(z)|= V_{C,K}(z), \ z\not\in \hat K.$$
\end{corollary}  

We remark that Bernstein-Markov measures exist in abundance; cf., \cite{BLPW}. Our goal in subsequent sections is to generalize the results in \cite{Bapp} and \cite {Bfam} of T. Bloom to give more constructive ways of recovering $ V_{C,K}$ from special families of polynomials.

\section{$C-$Robin function} We begin with the observation that a proof similar to that of Theorem 5.3.1 of \cite{K} yields the following result.

\begin{theorem}\label{cklimek} Let $C,C'\subset(\RR^+)^d$ be convex bodies and let $F:\CC^d\to \CC^d$ be a proper polynomial mapping satisfying
$$0<\liminf_{|z|\to \infty}  \frac{\sup_{J\in C} |[F(z)]^{J}|}{\sup_{J\in C'} |z^{J'}|} \leq \limsup_{|z|\to \infty}  \frac{\sup_{J\in C} |[F(z)]^J|}{\sup_{J'\in C'} |z^{J'}|} < \infty.$$
Then for $K\subset \CC^d$ compact,
$$V_{C,K}(F(z))=V_{C',F^{-1}(K)}(z).$$

\end{theorem} 

\begin{proof} Since $H_C(z):=\sup_{J\in C} \log |z^J|$, the hypothesis can be written  
 \begin{equation}\label{liminf}0<  \liminf_{|z|\to \infty} \frac{e^{H_{C}(F(z))}}{e^{H_{C'}(z)}} \leq \limsup_{|z|\to \infty} \frac{e^{H_{C}(F(z))}}{e^{H_{C'}(z)}}<\infty .\end{equation}
We first show that $\liminf_{|z|\to \infty} \frac{e^{H_{C}(F(z))}}{e^{H_{C'}(z)}} >0$ implies
\begin{equation}\label{lhineq} V_{C',F^{-1}(K)}(z)\leq V_{C,K}(F(z)).\end{equation}
Indeed, starting with $u\in L_{C'}$ with $u\leq 0$ on $F^{-1}(K)$, take 
$$v(z):= \sup u(F^{-1}(z))$$
where the supremum is over all preimages of $z$. Then $v\in PSH(\CC^d)$ and $v\leq 0$ on $K$. Note that $v(F(z))=u(z)$. 
Now $u\in L_{C'}$ implies
$$\limsup_{|z|\to \infty} [u(z)-H_{C'}(z)]\leq M <\infty.$$
To show $v\in L_{C}$, since $F$ is proper it suffices to show 
$$\limsup_{|z|\to \infty} [v(F(z))-H_{C}(F(z))]<\infty.$$
We have
$$\limsup_{|z|\to \infty} [v(F(z))-H_{C}(F(z))]$$
$$=\limsup_{|z|\to \infty} [v(F(z))- \ H_{C'}(z) +  \ H_{C'}(z)-H_{C}(F(z))]$$
$$\leq \limsup_{|z|\to \infty} [u(z)- \ H_{C'}(z)]-\liminf_{|z|\to \infty}[H_{C}(F(z))- \ H_{C'}(z)]$$
$$\leq M -\liminf_{|z|\to \infty}[H_{C}(F(z))- \ H_{C'}(z)]<\infty$$
from the hypothesized condition in (\ref{liminf}) so $v\in L_{C}$ and (\ref{lhineq}) follows.

Next we show that 
$$\limsup_{|z|\to \infty} \frac{e^{H_{C}(F(z))}}{e^{H_{C'}(z)}}<\infty \ \hbox{implies} \  V_{C',F^{-1}(K)}(z)\geq V_{C,K}(F(z)).$$
Letting $u\in L_{C}$ with $u\leq 0$ on $K$, we have $u(F(z))\in PSH(\CC^d)$ and $u(F(z))\leq 0$ on $F^{-1}(K)$ and we are left to show $u(F(z))\in L_{C'}$. Now 
$$\limsup_{|z|\to \infty}[u(F(z))-H_{C'}(z)]$$
$$=\limsup_{|z|\to \infty}[u(F(z))-H_{C}(F(z))+H_{C}(F(z))-H_{C'}(z)]$$
$$\leq \limsup_{|z|\to \infty}[u(F(z))-H_{C}(F(z))]+  \limsup_{|z|\to \infty}[H_{C}(F(z))-H_{C'}(z)]<\infty$$
from the hypothesized condition in (\ref{liminf}) and $u\in L_{C}$.

\end{proof}

We can apply this in $\CC^d$ with $C'=c \Sigma$ where $c\in \ZZ^+$ and $C$ is an arbitrary convex body in $(\RR^+)^d$. Given $K\subset \CC^d$ compact, provided we can find $F$ satisfying the hypotheses of Theorem \ref{cklimek}, from the relation
$$V_{C,K}(F(z))=V_{c\Sigma,F^{-1}(K)}(z)=cV_{F^{-1}(K)}(z)\in cL(\CC^d)$$
we can form a scaling of the standard Robin function (\ref{stdrobin}) for $V_{F^{-1}(K)}$, i.e., $\rho_{F^{-1}(K)}:=\rho_{\Sigma,F^{-1}(K)}$, and we have
$$c\rho_{F^{-1}(K)}(z)=\limsup_{|\lambda|\to \infty} [V_{C,K}(F(\lambda z))-c\log |\lambda|].$$
This gives a connection between the standard Robin function $\rho_{F^{-1}(K)}$ and something resembling a possible definition of a $C-$Robin function $\rho_{C,K}$ (the right-hand-side). Given $K\subset \CC^d$, the set $F^{-1}(K)$ can be very complicated so that, apriori, this relation has little practical value.

For the rest of this section, and for most of the subsequent sections, we work in $\CC^2$ with variables $z=(z_1,z_2)$ and we let $C$ be the triangle with vertices $(0,0), (b,0), (0,a)$ where $a,b$ are relatively prime positive integers. We recall from the introduction:
\begin{enumerate}
\item $H_C(z_1,z_2)=\max[\log^+|z_1|^b, \log^+|z_2|^a]$ (note $H_C=0$ on the closure of the unit polydisk $P^2:=\{(z_1,z_2): |z_1|,|z_2|< 1\}$);
\item defining $\lambda \circ (z_1,z_2):=(\lambda^az_1,\lambda^b z_2)$, we have 
\begin{equation} \label{2above} H_C(\lambda \circ (z_1,z_2))=H_C(z_1,z_2)+ab\log|\lambda|\end{equation} for $(z_1,z_2)\in \CC^2 \setminus P^2$ and $|\lambda| \geq 1$.
\end{enumerate}

\begin{definition} \label{crobtri} Given $u\in L_C$, we define the $C-$Robin function of $u$:  
$$\rho_u(z_1,z_2):=\limsup_{|\lambda|\to \infty} [u(\lambda \circ (z_1,z_2))-ab\log|\lambda|]$$
for $(z_1,z_2)\in \CC^2$.
\end{definition}

We claim that $\rho_u\in L_C$. To see this, we lift the circle action on $\CC^2$, 
$$\lambda \circ (z_1,z_2):=(\lambda^az_1,\lambda^b z_2),$$
to $\CC^3$ in the following manner:
$$\lambda \circ (t,z_1,z_2):=(\lambda t,\lambda^az_1,\lambda^b z_2).$$
Given a function $u\in L_C(\CC^2)$, we can associate a function $h$ on $\CC^3$ which satisfies
\begin{enumerate}
\item $h(1,z_1,z_2)=u(z_1,z_2)$ for all $(z_1,z_2)\in \CC^2$;
\item $h\in L_{\tilde C}(\CC^3)$ where $\tilde C=co\{(0,0,0),(1,0,0),(0,b,0),(0,0,a)\}$; 
\item $h$ is $ab-$log-homogeneous: 
$$h(\lambda \circ (t,z_1,z_2))=h(t,z_1,z_2)+ab \log |\lambda|.$$
\end{enumerate}

\noindent Indeed, we simply set
$$h(t,z_1,z_2):=u(\frac{z_1}{t^a},\frac{z_2}{t^b})+ab \log |t| \ \hbox{if} \ t\not = 0 \ \hbox{and}$$
$$h(0,z_1,z_2) :=\limsup_{(t,w_1,w_2)\to (0,z_1,z_2)} h(t,w_1,w_2)  \ \hbox{if} \ t = 0.$$
Now since $h$ is psh on $\CC^3$, we have
\begin{equation}\label{pshofrho} h(0,z_1,z_2) :=\limsup_{(t,z_1,z_2)\to (0,z_1,z_2)} h(t,z_1,z_2) = \rho_u(z_1,z_2). \end{equation} 

\begin{proposition}\label{prop:1.2} For $u\in L_C$, we have $\rho_u\in L_C$. In particular, $\rho_u$ is plurisubharmonic.
\end{proposition}

\begin{proof} The psh of $\rho_u$ follows directly from (\ref{pshofrho}) since $h$ is psh on $\CC^3$. To show $\rho_u\in L_C$, note that 
$$\rho_u(\lambda \circ (z_1,z_2))=\rho_u(z_1,z_2)+ab\log |\lambda| \ \hbox{for} \ \lambda \in \CC.$$
From (\ref{2above}) $H_C$ satisfies the same relation for $(z_1,z_2)\in \CC^2 \setminus P^2$ and $|\lambda| \geq 1$ which gives the result.
\end{proof}

\begin{remark} \label{rem43} Since $\rho_u(\lambda \circ (z_1,z_2))=\rho_u(z_1,z_2)+ab\log |\lambda|$, in particular, 
$$\rho_u(e^{i\theta} \circ (z_1,z_2))=\rho_u(z_1,z_2).$$
Moreover, any point $(z_1,z_2)\in \CC^2$ is of the form $(\lambda^a\zeta_1,\lambda^b\zeta_2)$ for some $(\zeta_1,\zeta_2)\in \partial P^2$ and some $\lambda \in \CC$. Indeed, if $b\geq a$ then we get all points $(z_1,z_2)\in \CC^2$ with $|z_2|^a \leq |z_1|^b$ as $(\lambda^a\zeta_1,\lambda^b\zeta_2)$ for some $(\zeta_1,\zeta_2)$ with $|\zeta_1|=1$ and $|\zeta_2|\leq 1$ and we get all points $(z_1,z_2)\in \CC^2$ with $|z_2|^a \geq |z_1|^b$ as $(\lambda^a\zeta_1,\lambda^b\zeta_2)$ for some $(\zeta_1,\zeta_2)$ with $|\zeta_1|\leq 1$ and $|\zeta_2|= 1$. Thus we recover the values of $\rho_u$ on $\CC^2$ from its values on $\partial P^2$.
\end{remark}

\begin{remark} \label{generald} In the general case where 
$$C=co\{(0,...,0),(a_1,0,...,0),...,(0,...,0,a_d)\}\in (\RR^+)^d$$ where $a_1,...,a_d$ are pairwise relatively prime, we have
$$H_C(z_1,...,z_d)=\max[a_j\log^+|z_j|:j=1,...,d]$$ and we define 
$$\lambda \circ (z_1,...,z_d):=(\lambda^{\prod_{j\not = 1}a_j}z_1,...,\lambda^{\prod_{j\not = d}a_j}z_d)$$ so that 
$$H_C(\lambda \circ (z_1,...,z_d))=H_C(z_1,...,z_d)+(\prod_{j=1}^d a_j)\log|\lambda|$$ for $(z_1,...,z_d)\in \CC^d \setminus P^d$ and $|\lambda| \geq 1$. Then given $u\in L_C$, we define the $C-$Robin function of $u$ as 
$$\rho_u(z_1,...,z_d):=\limsup_{|\lambda|\to \infty} [u(\lambda \circ (z_1,...,z_d))-(\prod_{j=1}^d a_j)\log|\lambda|]$$
for $(z_1,...,z_d)\in \CC^d$.

\end{remark}

We recall the Siciak-Zaharjuta formula (\ref{siczah}) for $K\subset \CC^d$ compact: 
$$V_{C,K}(z)=\sup\{\tfrac{1}{\deg_C(p)}\log|p(z)|\colon p\in\CC[z], \|p\|_K\leq 1\}$$ 
For simplicity in notation, we write $\rho_{C,K}:= \rho_{V_{C,K}^*}$. The following result will be used in section 8.

\begin{theorem} \label{circlek} Let $K\subset \CC^2$ be nonpluripolar and satisfy
\begin{equation}\label{circle} e^{i\theta}\circ K =K.\end{equation}
Then $K=\{\rho_{C,K}\leq 0\}$ and $V_{C,K}^*=\rho^+_{C,K}:=\max[\rho_{C,K},0]$.
\end{theorem}

\begin{proof} We first define a $C-$homogeneous extremal function $H_{C,K}$ associated to a general compact set $K$. 
To this end, for each $n\in\NN$ we define the collection of $nC$-homogeneous polynomials by
$$H_n(C):=\{h_n(z_1,z_2)=\sum_{(j,k): aj+bk=nab}c_{jk} z_1^jz_2^k: c_{jk}\in \CC\}\subset Poly(nC).$$
Note that for $h_n \in H_n(C)$,
$$h_n(\lambda \circ (z_1,z_2))= \lambda^{nab} \sum_{(j,k): aj+bk=nab}c_{jk} z_1^jz_2^k = \lambda^{nab}h_n(z_1,z_2)$$
and thus $u:=\frac{1}{n}\log |h_n|$ satisfies
\begin{equation}\label{lamb} u(\lambda \circ (z_1,z_2))= u(z_1,z_2)+ab\log |\lambda|.\end{equation}
Define 
\begin{equation}\label{hk}H_{C,K}(z_1,z_2):=\sup_n \sup \{\frac{1}{n}\log |h_n(z_1,z_2)|: h_n \in H_n(C), \ ||h_n||_K\leq 1\}.\end{equation}
Then $H_{C,K}$ satisfies the property in (\ref{lamb}). Clearly $$H_{C,K}^+:=\max[H_{C,K},0] \leq V_{C,K}$$ and hence $K\subset \{H_{C,K}\leq 0\}$.  

For a polynomial $p\in Poly(nC)$, we write 
\begin{equation}\label{polyhom}p(z_1,z_2)=\sum_{aj+bk\leq  nab} c_{jk} z_1^jz_2^k= \sum_{l=0}^{nab} \tilde h_l (z_1,z_2) \end{equation}
where $\tilde h_l (z_1,z_2):=\sum_{aj+bk=l} c_{jk} z_1^jz_2^k$ satisfies
$$\tilde h_l(\lambda \circ (z_1,z_2))=\lambda^l \tilde h_l (z_1,z_2).$$
Then for each $l=0,1,...,nab$,
\begin{equation}\label{hom} ||\tilde h_l||_K\leq ||p||_K.
\end{equation}
To prove (\ref{hom}), note that 
$$p(\lambda \circ (z_1,z_2))= \sum_{l=0}^{nab} \lambda^l \tilde h_l (z_1,z_2).$$
Take $(z_1,z_2)\in K$ at which $|\tilde h_l(z_1,z_2)|=||\tilde h_l||_K$. Then by the Cauchy estimates for $\lambda \to F(\lambda):= p(\lambda \circ (z_1,z_2))$ on the unit circle,
$$|\tilde h_l(z_1,z_2)|=||\tilde h_l||_K=|F^{(l)}(0)|/l! \ \leq \ \max_{|\lambda|=1}|F(\lambda)|\leq ||p||_K,$$
proving (\ref{hom}). 

We define
$$ \tilde H_l:=\{\tilde h_l(z_1,z_2):=\sum_{aj+bk=l} c_{jk} z_1^jz_2^k, \ c_{jk}\in \CC\}.$$
If $a=b=1$, $ \tilde H_l=H_l(C)=H_l(\Sigma)$ are the usual homogeneous polynomials of degree $l$ in $\CC^d$.
Moreover, if $\tilde h_l \in  \tilde H_l$, then $\tilde h_l^{ab} \in H_l(C)$. Since $||\tilde h_l||_K\leq 1$ if and only if $||\tilde h_l^{ab}||_K\leq 1$, this shows 
\begin{equation} \label{chomog} H_{C,K}(z_1,z_2)= ab\cdot \sup_l  \sup \{\frac{1}{l}\log |\tilde h_l(z_1,z_2)|: \tilde h_l \in \tilde H_l, \ ||\tilde h_l||_K\leq 1\}.\end{equation}

We define the $C-$homogeneous polynomial hull $\hat K_C$ of a compact set $K$ as
$$\hat K_C:=\{(z_1,z_2): |k(z_1,z_2)| \leq ||k||_K, \ k \in \cup_l \tilde H_l\}.$$
It is clear $\hat K \subset \hat K_C$ for any compact set $K$. We show the reverse inclusion, and hence equality, for $K$ satisfying (\ref{circle}). To this end, let $a\in \hat K_C$. For $p\in Poly(nC)$, write $p=\sum_{l=0}^{nab} \tilde h_l$ as in (\ref{polyhom}). Then 
$$|p(a)|\leq  \sum_{l=0}^{nab} |\tilde h_l(a)|\leq \sum_{l=0}^{nab} ||\tilde h_l||_K\leq (nab+1) ||p||_K.$$
Thus
$$|p(a)| \leq (nab+1) ||p||_K.$$
Apply this to $p^m\in Poly(nmC)$:
$$|p(a)|^m\leq (nmab+1)||p||_K^m \ \hbox{so that} \ |p(a)|^{1/n}\leq (nmab+1)^{1/nm}||p||_K^{1/n}.$$
Letting $m \to \infty$, we obtain $|p(a)|\leq ||p||_K$ and hence $a \in \hat K$.

We use this to show \begin{equation}\label{eqn:1.4} \{V_{C,K} =0\}=  \{H_{C,K}\leq 0\}\end{equation} for sets satisfying (\ref{circle}). To see this, we observe from (\ref{chomog}) that the right-hand-side of (\ref{eqn:1.4}) is the $C-$homogeneous polynomial hull $\hat K_C$ of $K$ while the left-hand-side is the polynomial hull $\hat K$ of $K$. Thus (\ref{eqn:1.4}) follows from the previous paragraph.

Now we claim that $V_{C,K}^*=H_{C,K}^+$. We observed that $H_{C,K}^+ \leq V_{C,K}\leq V_{C,K}^*$; for the reverse inequality, we observe that $H_{C,K}^+$ is in $L_C$ and since $H_{C,K}$ satisfies (\ref{lamb}), we have $H_{C,K}^+$ is maximal outside $\hat K$. From (\ref{eqn:1.4}) we can apply the global domination principle (Proposition \ref{gctp}) to conclude that $H_{C,K}^+ \geq V_{C,K}^*$ and hence $H_{C,K}^+ = V_{C,K}^*$. 

Using $H_{C,K}^+ = V_{C,K}^*$, 
$$\rho_{C,K}(z_1,z_2):=\limsup_{|\lambda|\to \infty} [H_{C,K}(\lambda \circ (z_1,z_2))-ab\log|\lambda|]$$
$$=\limsup_{|\lambda|\to \infty} H_{C,K}(z_1,z_2)=H_{C,K}(z_1,z_2)$$
for $(z_1,z_2)\in \CC^2\setminus K$ by the invariance of $H_{C,K}$ (i.e., it satisfies (\ref{lamb})). Thus, from Proposition \ref{prop:1.2} (and the invariance of $\rho_{C,K}$) we have 
$$\rho_{C,K}^+=H_{C,K}^+=V_{C,K}^*.$$
This shows $K=\{\rho_{C,K}\leq 0\}$ and $V_{C,K}^*=\rho^+_{C,K}:=\max[\rho_{C,K},0]$.
\end{proof}

\begin{remark} \label{foursix} It follows that for 
$$p=\sum_{l=0}^{nab} \tilde h_l =h_n + r_n \in Poly(nC)$$
where $h_n:=\tilde h_{nab}\in H_n(C)$ and $r_n=p-h_n = \sum_{aj+bk<nab} c_{jk} z_1^jz_2^k$ , if $u:=\frac{1}{n}\log |p_n|$ then
$$\rho_u = \frac{1}{n}\log |\tilde h_{nab}|  = \frac{1}{n}\log|h_n|.$$
We write $\hat p_n :=h_n = \tilde h_{nab}$; thus $\rho_u= \frac{1}{n}\log |\hat p_n|$.
\end{remark}

In the case $a=b=1$ where $C=\Sigma$, we know from Corollary 4.6 of \cite{BLM} that $K$ regular implies $\rho_K:=\rho_{\Sigma,K}$ is continuous. We need to know that for our triangles $C$ where $a,b$ are relatively prime positive integers we also have $\rho_{C,K}$ is continuous. To this end, we begin with the observation that applying Theorem \ref{cklimek} in the special case where $d=2$ and $C$ is our triangle with vertices $(0,0), (b,0), (0,a)$, we can take 
$$F(z_1,z_2)=(z_1^a,z_2^b)$$
and $c=ab$ to obtain
$$ab \rho_{F^{-1}(K)}(z_1,z_2)=\limsup_{|\lambda|\to \infty} [V_{C,K}(\lambda^az_1^a,\lambda^bz_2^b)-ab\log |\lambda|].$$
$$=\limsup_{|\lambda|\to \infty} [V_{C,K}(\lambda \circ (z_1^a,z_2^b))-ab\log |\lambda|]$$
$$=\rho_{C,K}(z_1^a,z_2^b)= \rho_{C,K}(F(z_1,z_2)).$$
We use this connection between $\rho_{C,K}$ and the standard Robin function $\rho_{F^{-1}(K)}$ to show that $\rho_{C,K}$ is continuous if $K$ is regular.

\begin{proposition}\label{robinreg} Let $K\subset \CC^2$ be compact and regular. Then $ \rho_{C,K}$ is uniformly continuous on $\partial P^2$.
\end{proposition}

\begin{proof} With $F(z_1,z_2)=(z_1^a,z_2^b)$ as above, from Theorem 5.3.6 of \cite{K}, we have $F^{-1}(K)$ is regular. Thus, from Corollary 4.6 of \cite{BLM}, $ \rho_{F^{-1}(K)}$ is continuous. Hence $\rho_{C,K}(z_1^a,z_2^b)= \rho_{C,K}(F(z_1,z_2))$ is continuous. To show $\zeta \to \rho_{C,K}(\zeta)$ is continuous at $\zeta =(\zeta_1,\zeta_2)\in \partial P^2$, we use the fundamental relationship that
$$ab \rho_{F^{-1}(K)}(z_1,z_2)=\rho_{C,K}(z_1^a,z_2^b).$$
To this end, let $\zeta^n =(\zeta^n_1,\zeta^n_2)\in \partial P^2$ converge to $\zeta =(\zeta_1,\zeta_2)$. Then 
$$\rho_{C,K}(\zeta^n_1,\zeta^n_2)=\rho_{C,K}([(\zeta^n_1)^{1/a}]^a,[(\zeta^n_2)^{1/b}]^b)$$
for any $a-$th root $(\zeta^n_1)^{1/a}$ of $\zeta^n_1$ and any $b-$th root $(\zeta^n_2)^{1/b}$ of $\zeta^n_2$. But 
$$\rho_{C,K}([(\zeta^n_1)^{1/a}]^a,[(\zeta^n_2)^{1/b}]^b)=ab \rho_{F^{-1}(K)}((\zeta^n_1)^{1/a},(\zeta^n_2)^{1/b}).$$
By continuity of $\rho_{F^{-1}(K)}$, 
$$\lim_{n\to \infty}\rho_{C,K}(\zeta^n_1,\zeta^n_2)= \lim_{n\to \infty} ab \rho_{F^{-1}(K)}((\zeta^n_1)^{1/a},(\zeta^n_2)^{1/b})$$
$$=ab \rho_{F^{-1}(K)}((\zeta_1)^{1/a},(\zeta_2)^{1/b})$$
for the appropriate choice of $(\zeta_1)^{1/a}$ and $(\zeta_2)^{1/b}$. But 
$$ab \rho_{F^{-1}(K)}((\zeta_1)^{1/a},(\zeta_2)^{1/b})=\rho_{C,K}( [(\zeta_1)^{1/a}]^a,[(\zeta_2)^{1/b}]^b)=\rho_{C,K}(\zeta_1,\zeta_2).$$
Note that this also yields that the value of $\rho_{F^{-1}(K)}((\zeta_1)^{1/a},(\zeta_2)^{1/b})$ is independent of the choice of the roots $(\zeta_1)^{1/a}$ and $(\zeta_2)^{1/b}$. This can also be seen from the definitions of $\rho_{F^{-1}(K)}$ and $F$.

\end{proof}

\begin{remark} \label{KEY} The relationship 
$$ab \rho_{F^{-1}(K)}(z_1,z_2)=\rho_{C,K}(z_1^a,z_2^b)$$
is a special case of a more general result. Let $u\in L_C$. Then 
$$\tilde u(z):=u(F(z_1,z_2))=u(z_1^a,z_2^b)\in abL=abL_{\Sigma} \ \hbox{and}$$
$$\rho_u(F(z_1,z_2))=\rho_u(z_1^a,z_2^b)=\limsup_{|\lambda|\to \infty}[u(\lambda\circ (z_1^a,z_2^b))-ab\log |\lambda|]$$
$$=\limsup_{|\lambda|\to \infty}[u(\lambda^a z_1^a,\lambda^bz_2^b)-ab\log |\lambda|]$$
$$=\limsup_{|\lambda|\to \infty}[\tilde u(\lambda z)-ab\log |\lambda|].$$
Since $\tilde u\in abL$, this last line is equal to the ``usual'' Robin function of $\tilde u$ in the sense of (\ref{stdrobin}). To be precise, it is equal to $ab{\bf \rho_{\tilde u/ab}}$ where ${\bf \rho_{\tilde u/ab}}$ is the standard Robin function (\ref{stdrobin}) of $\tilde u/ab\in L$. This observation will be crucial in section 6.

\end{remark}

We need an analogue of formula (18) in \cite{Zer} in order to verify a calculation in the next section. We follow the arguments in \cite{Zer}. Recall we may lift the circle action on $\CC^2$ to $\CC^3$ via
$$\lambda \circ (t,z_1,z_2):=(\lambda t,\lambda^az_1,\lambda^b z_2).$$
This gave a correspondence between $L_C(\CC^2)$ and $L_{\tilde C}(\CC^3)$ where $\tilde C=co\{(0,0,0),(1,0,0),(0,b,0),(0,0,a)\}$. 
In analogy with our class 
$$H_n(C):=\{h_n(z_1,z_2)=\sum_{(j,k): aj+bk=nab}c_{jk} z_1^jz_2^k: c_{jk}\in \CC\}\subset Poly(nC)$$
in $\CC^2$, we can consider
$$H_n(\tilde C):= \{h_n(t,z_1,z_2)=\sum_{(i,j,k): i+aj+bk=nab}c_{ijk} t^i z_1^jz_2^k: c_{ijk}\in \CC\}\subset Poly(n\tilde C)$$ in $\CC^3$. For $h_n\in H_n(\tilde C)$, we have 
$$u_n(t,z_1,z_2):=\frac{1}{n} \log |h_n(t,z_1,z_2)|$$
belongs to $L_{\tilde C}(\CC^3)$ and $u_n$ is $ab-$log-homogeneous. That $u_n \in L_{\tilde C}(\CC^3)$ is clear; to show the $ab-$log-homogeneity, note that
$$h_n(\lambda \circ (t,z_1,z_2))= h_n(\lambda t,\lambda^az_1,\lambda^b z_2)$$
$$= \sum_{(i,j,k): i+aj+bk=nab}c_{ijk} (\lambda t)^i (\lambda^az_1)^j (\lambda^b z_2)^k$$
$$=\sum_{(i,j,k): i+aj+bk=nab}c_{ijk} \lambda^{i+aj+bk}  t^i z_1^jz_2^k= \lambda^{nab} h_n(t,z_1,z_2)$$
so that
$$u_n(\lambda \circ (t,z_1,z_2))=u_n(t,z_1,z_2)+ ab\log |\lambda|.$$
Moreover, for $h_n\in H_n(\tilde C)$, the polynomial 
$$p_n(z_1,z_2):= h_n(1,z_1,z_2)=\sum_{(j,k): aj+bk \leq nab}c_{ijk} z_1^jz_2^k \in Poly(nC);$$
conversely, if $p_n(z_1,z_2)=\sum_{(j,k): aj+bk \leq nab}c_{jk} z_1^jz_2^k \in Poly(nC)$ then 
$$h_n(t,z_1,z_2):= t^{nab}\cdot p_n(\frac{z_1}{t^a},\frac{z_2}{t^b})\in H_n(\tilde C).$$

Next, given a compact set $E\subset \CC^3$, we define the $ab-$log-homogeneous $\tilde C-$extremal function
$$H_{\tilde C,E}(t,z_1,z_2):= \sup\{\tfrac{1}{\deg_{\tilde C}(p)}\log|p(t,z_1,z_2)|\colon p\in \cup_n H_n(\tilde C), \|p\|_E\leq 1\}$$
and its usc regularization $H_{\tilde C,E}^*$. Given the one-to-one correspondence between $Poly(nC)$ in $\CC^2$ and $H_n(\tilde C)$ in $\CC^3$, we see that for $K\subset \CC^2$ compact,
\begin{equation}\label{hk=vck} V_{C,K}(z_1,z_2)=H_{\tilde C,\{1\}\times K}(1,z_1,z_2) \ \hbox{for all} \ (z_1,z_2)\in \CC^2\end{equation}
and hence a similar equality holds for the usc regularizations of both sides. Using this, we observe that for $\zeta=(\zeta_1,\zeta_2)\not=(0,0)$, we have
$$\rho_{C,K}(\zeta)=\limsup_{|\lambda|\to \infty}[V_{C,K}^*(\lambda \circ \zeta)-ab\log |\lambda|]$$
$$=\limsup_{|\lambda|\to \infty}[H_{\tilde C,\{1\}\times K}^*(1,\lambda^a \zeta_1,\lambda^b \zeta_2) -ab\log |\lambda|]$$
$$= \limsup_{|\lambda|\to \infty}H_{\tilde C,\{1\}\times K}^*(1/\lambda,\zeta_1,\zeta_2) =H_{\tilde C,\{1\}\times K}^*(0,\zeta_1,\zeta_2).$$
Here we have used the fact that
$$H_{\tilde C,\{1\}\times K}^*(\lambda \circ(1/\lambda,\zeta_1,\zeta_2))=H_{\tilde C,\{1\}\times K}^*(1,\lambda^a \zeta_1,\lambda^b \zeta_2) .$$
We state this as a proposition:

\begin{proposition}\label{propneed} For $K\subset \CC^2$ compact, 
$$\rho_{C,K}(\zeta_1,\zeta_2)=H_{\tilde C,\{1\}\times K}^*(0,\zeta_1,\zeta_2) \ \hbox{for all} \ (\zeta_1,\zeta_2)\not=(0,0).$$

\end{proposition}

\begin{remark} Using the relation (\ref{hk=vck}) and following the reasoning in \cite{Siciak}, Proposition 2.3, it follows that a compact set $K\subset \CC^2$ is regular; i.e., $V_{C,K}$ is continuous in $\CC^2$, if and only if $H_{\tilde C,\{1\}\times K}$ is continuous in $\CC^3$. Thus we get an alternate proof of Proposition \ref{robinreg}.

\end{remark}

\section{Preliminary results: Triangle case} We continue to let $C$ be the triangle with vertices at $(0,0), (b,0)$, and  $(0,a)$ where $a,b$ are relatively prime positive integers. For $K\subset \CC^2$ compact and $\zeta:=(\zeta_1,\zeta_2)\in \partial P^2$, we define Chebyshev constants
$$\kappa_n:=\kappa_n(K,\zeta):=\inf\{||p_n||_K: p_n\in Poly(nC), \ |\hat p_n(\zeta)|=1\}.$$
We note that $\kappa_{n+m}\leq \kappa_n \kappa_m$: if we take $t_n,t_m$ achieving $\kappa_n,\kappa_m$, then 
$t_nt_m\in Poly(n+m)C$ and $\hat{t_nt_m}=\hat t_n \hat t_m$ (see Remark \ref{foursix}) so that $$\kappa_{n+m}\leq ||t_nt_m||_K\leq \kappa_n \kappa_m.$$
Thus $\lim_{n\to \infty} \kappa_n^{1/n}$ exists and we set
$$\kappa(K,\zeta)= \lim_{n\to \infty} \kappa_n^{1/n}=\inf_{n} \kappa_n^{1/n}.$$
The following relation between $\kappa(K,\zeta)$ and $\rho_{C,K}(\zeta)$ is analogous to Proposition 4.2 of \cite{Niv}.

\begin{proposition} \label{niv} For $\zeta \in \partial P^2$,
$$\kappa(K,\zeta)= e^{-\rho_{C,K}(\zeta)}.$$
\end{proposition}

\begin{proof} We first note that
$$\kappa_n(K,\zeta)=\inf\{\frac{||p_n||_K}{|\hat p_n(\zeta)|}: p_n\in Poly(nC)\}$$
$$=\inf\{\frac{1}{|\hat p_n(\zeta)|}: p_n\in Poly(nC), \ ||p_n||_K\leq 1\}.$$
Thus for any $p_n\in Poly(nC)$ with $||p_n||_K\leq 1$, $\kappa_n(K,\zeta)\leq \frac{1}{|\hat p_n(\zeta)|}$. For such $p_n$, $\frac{1}{n}\log |p_n(z)|\leq V_{C,K}(z)$ for all $z\in \CC^2$ so that
$$\frac{1}{n}\log |\hat p_n(\zeta)|\leq \rho_{C,K}(\zeta); \ \hbox{i.e.} \ \frac{1}{|\hat p_n(\zeta)|^{1/n}}\geq  e^{-\rho_{C,K}(\zeta)}$$
for all $\zeta\in \partial P^2$. Taking the infimum over all such $p_n$,
$$\kappa_n(K,\zeta)^{1/n}\geq e^{-\rho_{C,K}(\zeta)}$$
for all $n$; taking the limit as $n\to \infty$ gives
$$\kappa(K,\zeta)\geq e^{-\rho_{C,K}(\zeta)}.$$

To prepare for the reverse inequality, we let $\{b_j\}$ be an orthonormal basis of $\bigcup_n Poly(nC)$ in $L^2(\mu)$ where 
$\mu$ is any Bernstein-Markov measure for $K$: thus for any $\epsilon >0$, there exists a constant $c_{\epsilon}$ so that
$$||p_n||_K\leq c_{\epsilon}(1+\epsilon)^{\deg_C(p_n)} ||p_n||_{L^2(\mu)}, \ p_n\in Poly(nC), \ n=1,2,...$$
In particular,
$$||b_j||_K \leq c_{\epsilon}(1+\epsilon)^{\deg_C(b_j)}$$ 
and from Corollary \ref{bmext},
\begin{equation}\label{corbm} \limsup_{j \to \infty} \tfrac{1}{\deg_C(b_j)}\log|b_j(z)|= V_{C,K}(z), \ z\not\in \hat K.\end{equation}

We next show that for $\zeta \in \partial P^2$,
\begin{equation}\label{keyrobin} \limsup_{j \to \infty} \tfrac{1}{\deg_C(b_j)}\log|\hat b_j(\zeta)|= \rho_{C,K}(\zeta).\end{equation}
For one inequality, we use the fact that for a function $u$ subharmonic on $\CC$ with $u\in L$, the function $r\to \max_{|t|=r}u(t)$ is a convex function of $\log r$. Hence 
$$\limsup_{|t|\to \infty} [u(t)-\log |t|]=\inf_r \bigl(\max_{|t|=r}u(t)-\log r\bigr).$$ 
Thus if $u\in abL$; i.e., $u(z)-ab\log |z| =0(1), \ |z|\to \infty$, we have
\begin{equation}\label{uinabl} \limsup_{|t|\to \infty} [u(t)-ab\log |t|]=\inf_r\bigl(\max_{|t|=r}u(t)-ab\log r\bigr).\end{equation} 
Fix $\zeta \in \partial P^2$ and letting $d_j:=\deg_C(b_j)$ apply this to the function 
$$\lambda \to \frac{1}{d_j}\log |b_j(\lambda \circ \zeta)|=\frac{1}{d_j}\log |b_j(\lambda^a\zeta_1, \lambda^b \zeta_2)|.$$
We obtain (using also Remark \ref{foursix}), for any $r$,
$$\frac{1}{d_j}\log |\hat b_j(\zeta)|=\limsup_{|\lambda|\to \infty} [ \frac{1}{d_j}\log |b_j(\lambda \circ \zeta)|-ab\log |\lambda|]$$
$$\leq \max_{|\lambda|=r} \frac{1}{d_j}\log |b_j(\lambda \circ \zeta)| -ab\log r.$$
Thus
$$\limsup_{j\to \infty} \frac{1}{d_j}\log |\hat b_j(\zeta)|\leq 
\limsup_{j\to \infty} \bigl(\max_{|\lambda|=r} \frac{1}{d_j}\log |b_j(\lambda \circ \zeta)| -ab\log r\bigr)$$
$$\leq \max_{|\lambda|=r} \bigl(\limsup_{j\to \infty}\frac{1}{d_j}\log |b_j(\lambda \circ \zeta)| -ab\log r\bigr)=\max_{|\lambda|=r} [V_{C,K}(\lambda \circ \zeta)-ab\log r]$$
where we used Hartogs lemma and (\ref{corbm}). Thus, letting $r\to \infty$,
$$\limsup_{j\to \infty} \frac{1}{d_j}\log |\hat b_j(\zeta)|\leq \rho_{C,K}(\zeta).$$

In order to prove the reverse inequality in (\ref{keyrobin}), we use Proposition \ref{propneed}. With the notation from the previous section, and following the proof of Th\'eor\`eme 2 in \cite{Zer}, let $h \in H_n(\tilde C)$ with $||h||_{1\times K}\leq 1$. Then 
$$p(z_1,z_2):=h(1,z_1,z_2) \in Poly(nC)$$ with $||p||_K\leq 1$. Writing $p=\sum_{j=1}^{N_n} c_j b_j$ where $N_n=$dim$(Poly(nC))$ as in the proof of Proposition \ref{threeone}, we have $|c_j|\leq 1$ and hence
$$|h(1,z_1,z_2)|=|p(z_1,z_2)|\leq \sum_{j=1}^{N_n} |b_j(z_1,z_2)|.$$
Then
$$|h(1/\lambda,z_1,z_2)|=|\lambda^{-nab}\cdot p_n(\lambda^a z_1,\lambda^b z_2)|\leq |\lambda^{-nab}\cdot \sum_{j=1}^{N_n} b_j(\lambda^a z_1,\lambda^b z_2)|.$$
Fixing $(z_1,z_2)=(\zeta_1,\zeta_2)$ and letting $|\lambda|\to \infty$, we get
$$|h(0,\zeta_1,\zeta_2)|\leq \sum_{b_j \in PolynC \setminus Poly(n-1)C} |\hat b_j(\zeta_1,\zeta_2)|\leq (N_n-N_{n-1}) |\hat b_{j_n}(\zeta_1,\zeta_2)|$$
where $N_{n-1}\leq j_n \leq N_n$. Using Proposition \ref{propneed} we conclude that 
$$\rho_{C,K}(\zeta_1,\zeta_2)\leq \limsup_{j \to \infty} \tfrac{1}{\deg_C(b_j)}\log|\hat b_j(\zeta_1,\zeta_2)|$$
and (\ref{keyrobin}) is proved.

We now use (\ref{keyrobin}) to prove that $\kappa(K,\zeta)\leq e^{-\rho_{C,K}(\zeta)}$ for $\zeta \in \partial P^2$ which will finish the proof of the proposition. Fixing such a $\zeta$ and $\epsilon >0$, take a subsequence $\{b_{k_j}\}$ with $d_j:=\deg_C(b_{k_j})$ such that
$$\frac{1}{d_j}\log |\hat b_{k_j}(\zeta)|\geq \rho_{C,K}(\zeta)-\epsilon, \ j\geq j_0(\epsilon).$$
Letting 
$$p_j(z):=\frac{b_{k_j}(z)}{c_{\epsilon}(1+\epsilon)^{d_j}},$$
we have $||p_j||_K\leq 1$ and
$$\rho_{C,K}(\zeta)-\epsilon \leq \frac{1}{d_j}\log |\hat b_{k_j}(\zeta)|=\frac{1}{d_j}\log |\hat p_j(\zeta)|+\frac{1}{d_j}\log c_{\epsilon} +\log (1+\epsilon).$$
Thus
$$\epsilon -\rho_{C,K}(\zeta)\geq \frac{1}{d_j}\log \frac{1}{|\hat p_j(\zeta)|}-\frac{1}{d_j}\log c_{\epsilon} -\log (1+\epsilon)$$
$$\geq \frac{1}{d_j}\log \kappa_{d_j}(K,\zeta)-\frac{1}{d_j}\log c_{\epsilon} -\log (1+\epsilon).$$
Letting $j\to \infty$,
$$\epsilon -\rho_{C,K}(\zeta)\geq \log \kappa (K,\zeta)-\log (1+\epsilon);$$
which holds for all $\epsilon >0$. Letting $\epsilon \to 0$ completes the proof.

\end{proof}

Using this proposition, and the observation within its proof that
$$\kappa_n(K,\zeta)=\inf\{\frac{1}{|\hat p_n(\zeta)|}: p_n\in Poly(nC), \ ||p_n||_K\leq 1\},$$
we obtain a result which will be useful in proving Theorem \ref{sicbloom}.

\begin{corollary} \label{siciak2.3} Let $K\subset \CC^2$ be compact and regular. Given $\epsilon >0$, there exists a positive integer $m$ and a finite set of polynomials $\{W_1,...,W_s\}\subset Poly(mC)$ such that $||W_j||_K=1, \ j=1,...,s$ and 
$$\frac{1}{m}\log \max_j |\hat W_j(\zeta)|\geq \rho_{C,K}(\zeta)-\epsilon \ \hbox{for all} \ \zeta \in \partial P^2.$$

\end{corollary}

\begin{proof} From Proposition \ref{niv}, given $\epsilon >0$, for each $\zeta \in \partial P^2$ we can find a polynomial $p\in Poly(nC)$ for $n\geq n_0(\epsilon)$ with $||p||_K=1$ and 
$$\frac{1}{n}\log |\hat p(\zeta)|\geq  \rho_{C,K}(\zeta)-\epsilon.$$
By continuity of $\rho_{K,C}$, which follows from Proposition \ref{robinreg}, such an inequality persists in a neighborhood of $\zeta$. We take a finite set $\{p_1,...,p_s\}$ of such polynomials with $p_i\in Poly(n_iC)$ such that 
$$\max_i \frac{1}{n_i}\log |\hat p_i(\zeta)|\geq  \rho_{C,K}(\zeta)-\epsilon  \ \hbox{for all} \ \zeta \in \partial P^2.$$
Raising the $p_i$'s to powers to obtain $W_i$'s of the same $C-$degree $m$, we still have $||W_i||_K=1$ and 
$$\frac{1}{m}\log \max_j |\hat W_j(\zeta)|\geq \rho_{C,K}(\zeta)-\epsilon \ \hbox{for all} \ \zeta \in \partial P^2.$$
\end{proof}

Given $K\subset \CC^2$ compact, and given $h_n \in H_n(C)$, we define 
$$Tch_Kh_n:= h_n + p_{n-1} \ \hbox{where} \ p_{n-1}\in Poly(n-1)C$$
and $||Tch_Kh_n||_K =\inf \{||h_n + q_{n-1}||_K: q_{n-1}\in Poly(n-1)C\}$. The polynomial $Tch_Kh_n$ need not be unique but each such polynomial yields the same value of $||Tch_Kh_n||_K $. The next result is similar to Theorem 3.2 of \cite{Bapp}.

\begin{theorem}\label{sicbloom} Let $K\subset \CC^2$ be compact, regular and polynomially convex. If $\{Q_n\}$ is a sequence of polynomials with $Q_n\in H_n(C)$ satisfying
$$\limsup_{n\to \infty} \frac{1}{n}\log |Q_n(\zeta)|\leq \rho_{C,K}(\zeta), \ \hbox{all} \ \zeta \in \partial P^2,$$
then
$$\limsup_{n\to \infty} ||Tch_KQ_n||_K^{1/n} \leq 1.$$
\end{theorem} 

\begin{proof} We follow the proof in \cite{Siciak}. Given $\epsilon >0$, we start with polynomials $\{W_1,...,W_s\}\subset Poly(mC)$ such that $||W_j||_K=1, \ j=1,...,s$ and 
$$\frac{1}{m}\log \max_j |\hat W_j(\zeta)|\geq \rho_{C,K}(\zeta)-\epsilon \ \hbox{for all} \ \zeta \in \partial P^2$$
(and hence on all of $\CC^2$) from Corollary \ref{siciak2.3}. From the hypotheses on $\{Q_n\}$ and the continuity of $\rho_{C,K}$ (Proposition \ref{robinreg}), we apply Hartogs lemma to conclude
$$\frac{1}{n}\log |Q_n(\zeta)| < \rho_{C,K}(\zeta)+\epsilon, \ \zeta \in \partial P^2, \ n\geq n_0(\epsilon).$$
Thus
\begin{equation} \label{no5} \frac{1}{n}\log |Q_n(\zeta)| < \frac{1}{m}\log \max_j |\hat W_j(\zeta)|+2\epsilon,  \ \zeta \in \partial P^2, \ n\geq n_0(\epsilon).\end{equation}
Note that $Q_n\in H_n(C)$ implies $Q_n(\lambda \circ \zeta)=\lambda^{nab}Q_n(\zeta)$ so that 
$$\frac{1}{n}\log |Q_n(\lambda \circ \zeta)| =\frac{1}{n}\log |Q_n(\zeta)| +ab\log |\lambda|.$$
Similary $\hat W_j \in H_m(C)$ implies 
$$\frac{1}{m}\log |\hat W_j(\lambda \circ \zeta)| =\frac{1}{m}\log |\hat W_j(\zeta)| +ab\log |\lambda|$$
so that (\ref{no5}) holds on all of $\CC^2$. 

We fix $R>1$ and define
$$G:=\{z\in \CC^2: |\hat W_j(z)| < R^m, \ j=1,...,s\}.$$
Since $\hat W_j(\lambda \circ \zeta)=\lambda^{mab}\hat W_j(\zeta)$, we have $e^{i\theta}\circ G=G$; since $\hat W_j(0)=0$, we have $0\in G$. We claim $G$ is bounded. To see this, choose $r>0$ so that 
$$K \subset rP^2 =\{(z_1,z_2)\in \CC^2: |z_1|, |z_2|\leq r\}.$$
Then $V_{C,K}(z_1,z_2)\geq H_C(z_1/r,z_2/r)$ and hence
$$\rho_{C,K}(z_1,z_2) \geq H_C(z_1/r,z_2/r), \ (z_1,z_2)\in \CC^2\setminus rP^2.$$
Since
$$\frac{1}{m}\log \max_j |\hat W_j(z_1,z_2)|\geq \rho_{C,K}(z_1,z_2)-\epsilon \ \hbox{for all} \ (z_1,z_2)\in \CC^2,$$
$G$ is bounded.

Next, choose $\delta >0$ sufficiently large so that 
$$K\cup G\subset \{(z_1,z_2)\in \CC^2: |z_1|,|z_2|< \delta R^m\}.$$
Define
$$\Delta:= \{(z,w)\in \CC^2\times \CC^s: |z_1|,|z_2|< \delta R^m, \ |w_j|< \delta R^m, \ j=1,...,s\}.$$
Given $\theta >1$, we can choose $p>0$ sufficiently large so that 
$$D:=\{(z,w)\in \CC^2\times \CC^s:|z_1|^p+|z_2|^p+|w_1|^p+\cdots+|w_s|^p < (\theta^{abm} \delta R^m)^p\},$$
which is complete circled (in the ordinary sense) and strictly pseudoconvex, satisfies 
$$\Delta \subset D \subset \theta^{abm} \Delta$$
(note this is just a replacement of an $l^{\infty}-$norm with an $l^p-$norm). 

We write $z:=(z_1,z_2)\in \CC^2$ and $\hat W(z):=(\hat W_1(z),...,\hat W_s(z))\in \CC^s$ for simplicity in notation. Let 
$$Y:=\{(z,\delta \hat W(z))\in \CC^2\times \CC^s: z\in \CC^2\}.$$
Then $Y$ is a closed, complex submanifold of $\CC^2\times \CC^s$. Appealing to the bounded, holomorphic extension result stated as Theorem 3.1 in \cite{Siciak}, there exists a positive constant $M$ such that for every $f \in H^{\infty}(Y\cap D)$ there exists $F\in H^{\infty}(D)$ with 
$$||F||_D \leq M ||f||_{Y\cap D} \ \hbox{and} \ F=f \ \hbox{on} \ Y\cap D.$$
We will apply this to the polynomials $Q_n(z)$. First, we observe that if $\pi: \CC^2\times \CC^s\to \CC^2$ is the projection $\pi(z,w)=z$, then
$$G =\pi (Y\cap \Delta) \subset \pi(Y\cap D) \subset \pi (Y\cap \theta^{abm} \Delta) \subset \theta \circ G.$$
To see the last inclusion -- note we use $\theta \circ G$, not $\theta G$ -- first note that 
$$s=\theta \circ z \iff z= \frac{1}{\theta}\circ s$$
and thus since $\hat W_j( \frac{1}{\theta}\circ s)=  \frac{1}{\theta^{mab}}\hat W_j(s)$, 
$$\theta \circ G=\{z\in \CC^2: |\hat W_j(z)| < (\theta^{ab}R)^m, \ j=1,...,s\}.$$
On the other hand,
$$\pi (Y\cap \theta^{abm} \Delta)=\{z\in \CC^2: (z,w)\in Y\cap \theta^{abm} \Delta\}$$
$$=\{z\in \CC^2:|z_1|, |z_2|< \theta^{abm} \delta R^m, \  \delta |\hat W_j(z)|< \theta^{abm}\delta R^m, \ j=1,...,s\}.$$

Applying the bounded holomorphic extension theorem to $f(z,w):=Q_n(z)$ for each $n$, we get $F_n(z,w)\in H^{\infty}(D)$ with 
$$Q_n(z)=F_n (z, \delta \hat W(z))$$
for all $z\in \pi(Y\cap D)$ and
$$||F_n||_D \leq M ||Q_n||_{\pi(Y\cap D)}.$$
Utilizing the set inclusion $\pi(Y\cap D) \subset \theta \circ G$, the definition of $\theta \circ G$ and (\ref{no5}) (which recall is valid on all of $\CC^2$), 
\begin{equation}\label{no7} ||F_n||_D \leq M ||Q_n||_{\theta \circ G}\leq M(e^{2\epsilon} \theta^{ab}R)^n \ \hbox{for} \ n\geq n_0(\epsilon). \end{equation}

Since $D$ is complete circled, we can expand $F_n$ into a series of homogeneous polynomials which converges locally uniformly on all of $D$. Rearranging into a multiple power series, we write
$$F_n(z,w):=\sum_{|I|+|J|\geq 0} a_{IJ}z^Iw^J, \ (z,w)\in D.$$
Using $Q_n(z)=F_n (z, \delta \hat W(z))$ for $z\in \pi(Y\cap D)$, we obtain for such $z$,
$$Q_n(z)={\sum}' a_{IJ}z^I (\delta \hat W(z))^J$$
where the prime denotes that the sum is taken over multiindices 
\begin{equation}\label{prime} I=(i_1,i_2)\in (\ZZ^+)^2, \ J\in (\ZZ^+)^s, \ \hbox{where} \ ai_1+bi_2 =:iab \ \hbox{and} \ i+|J|m=n.\end{equation}
This is because $Q_n\in H_n(C)$ and $\hat W_j\in H_m(C), \ j=1,...,s$. Precisely, each $\hat W_j(z)$ is of the form $\sum_{a\alpha + b\beta = mab} c_{\alpha \beta} z_1^{\alpha}z_2^{\beta}$ so that if $J=(j_1,...,j_s)$, a typical monomial occurring in $Q_n(z)$ must be of the form
\begin{equation}\label{monomial} z_1^{i_1}z_2^{i_2}(z_1^{\alpha_1}z_2^{\beta_1})^{j_1}\cdots (z_1^{\alpha_s}z_2^{\beta_s})^{j_s}\end{equation}
where $a\alpha_k +b\beta_k =mab, \ k=1,...,s$; hence
$$a(\alpha_1j_1+\cdots +\alpha_s j_s) + b(\beta_1j_1+\cdots +\beta_s j_s) =|J|mab.$$
In order for (\ref{monomial}) to (possibly) appear in $Q_n(z)$, we require (\ref{prime}). The positive integers $i$ in (\ref{prime}) are related to the lengths $|I|$ by $|I|=i_1+i_2\leq ai_1+bi_2=iab$; and if, say, $a\leq b$ we have a reverse estimate
$$iab= ai_1+bi_2 \leq b|I| \ \hbox{so that} \ |I|\geq ia.$$
However, all we will need to use is the fact that the number of multiindices occurring in the sum for $Q_n(z)$ is at most $N_n=$ dim$(Poly(nC))$ and $\lim_{n\to \infty} N_n^{1/n}=1$.

Applying the Cauchy estimates on the polydisk $\Delta \subset D$, we obtain
\begin{equation} \label{no8} |a_{IJ}|\leq \frac{||F_n||_D}{(\delta R^m)^{|I|+|J|}}\end{equation}
for $n\geq n_0(\epsilon)$. 

We now define
$$p_n(z):={\sum}' a_{IJ}z^I (\delta W(z))^J.$$
From (\ref{prime}) and the previously observed fact that 
$$\hbox{if} \ q_j \in Poly(n_jC), \ j=1,2 \ \hbox{then} \ \hat {q_1q_2}=\hat q_1 \hat q_2,$$ we have $\hat p_n(z) = Q_n(z)$. Using the estimates (\ref{no7}), (\ref{no8}), the facts that $||W_j||_K=1, \ j=1,...,s$ and 
$$K\cup G\subset \{(z_1,z_2)\in \CC^2: |z_1|,|z_2|< \delta R^m\},$$
we obtain
$$||Tch_K Q_n||_K \leq ||p_n||_K \leq {\sum}' \frac{M(e^{2\epsilon} \theta^{ab}R)^n}{(\delta R^m)^{|I|+|J|}}\cdot (\delta R^m)^{|I|}\delta^{|J|}$$
$$={\sum}' M(e^{2\epsilon} \theta^{ab})^n\cdot R^{n-m|J|}\leq {\sum}' M(e^{2\epsilon} \theta^{ab})^n\cdot R^n$$
$$\leq C_n (e^{2\epsilon} \theta^{ab}R)^n, \ n\geq n_0(\epsilon),$$
where $C_n$ can be taken as $M$ times the cardinality of the set of multiindices in (\ref{prime}). Clearly $\lim_{n\to \infty} C_n^{1/n}=1$ so that
$$ \limsup_{n\to \infty} ||Tch_KQ_n||_K^{1/n} \leq e^{2\epsilon} \theta^{ab}R.$$
Since $\epsilon >0, \ R>1$ and $\theta >1$ were arbitrary, the result follows.

\end{proof}

\section{The integral formula}

In the standard setting of the Robin function ${\bf \rho_u}$ associated to $u\in L(\CC^2)$ (cf., \ref{stdrobin}), for $z=(z_1,z_2)\not = (0,0)$ we can define
$$ \underline {\bf \rho_u}(z):=\limsup_{|\lambda|\to \infty} [u(\lambda z)-\log |\lambda z|]={\bf \rho_u}(z)-\log |z| $$
so that $\underline {\bf \rho_u}(tz)=\underline {\bf \rho_u}(z)$ for $t\in \CC \setminus \{0\}$. Thus we can consider $\underline {\bf \rho_u}$ as a function on $\PP^1$ where to $p=(p_1,p_2)\in \partial P^2$ we associate the point where the complex line $\lambda \to \lambda p$ hits $H_{\infty}$. The integral formula Theorem 5.5 of \cite{BT} in this setting is the following.

\begin{theorem} \label{bt5.5} {\bf(Bedford-Taylor)} Let $u,v,w\in L^+(\CC^2)$. Then
$$\int_{\CC^2} (udd^cv-vdd^cu)\wedge dd^cw= 2\pi  \int_{\PP^1} (\underline {\bf \rho_u}-\underline {\bf \rho_v})\wedge (dd^c \underline{\bf \rho_w}+\Omega)$$
where $\Omega$ is the standard K\"ahler form on $\PP^1$.
\end{theorem}

We use this to develop an integral formula for $u,v,w\in L_{C,+}$. Letting 
$$ \zeta=(\zeta_1,\zeta_2)=F(z)=F(z_1,z_2)=(z_1^a,z_2^b), $$ 
we recall that for $u\in L_C$, we have 
\begin{equation}\label{raheq} \tilde u(z):=u(F(z_1,z_2))=u(\zeta) \in abL \ \hbox{and}\end{equation}
\begin{equation}\label{roheq} \rho_u(\zeta)=\rho_u(F(z_1,z_2))=ab{\bf \rho_{\tilde u/ab}}(z)\end{equation} 
where ${\bf \rho_{\tilde u/ab}}$ is the standard Robin function of $\tilde u/ab\in L$. It follows from the calculations in Remark \ref{KEY} that if $u\in L_{C,+}$ then $\tilde u \in abL^+$. From (\ref{raheq}), if $u,v,w\in L_{C,+}$,
$$\int_{\CC^2} (udd^cv-vdd^cu)\wedge dd^cw = \int_{\CC^2} (\tilde udd^c \tilde v-\tilde vdd^c\tilde u)\wedge dd^c\tilde w.$$
We apply Theorem \ref{bt5.5}  to the right-hand-side, multiplying by factors of $ab$ since $\tilde u, \tilde v, \tilde w\in abL^+$, to obtain the desired integral formula:
\begin{equation}\label{riheq} \int_{\CC^2} (udd^cv-vdd^cu)\wedge dd^cw = 2\pi (ab)^3 \int_{\PP^1} (\underline {\bf \rho_{\tilde u/ab}}-\underline {\bf \rho_{\tilde v/ab}})\wedge (dd^c \underline {\bf \rho_{\tilde w/ab}}+ \Omega).\end{equation}

\begin{corollary} \label{maincor} Let $u,v \in L_{C,+}$ with $u\geq v$. Then
$$\int_{\CC^2} u(dd^cv)^2\leq \int_{\CC^2} v(dd^cu)^2$$
$$+2\pi (ab)^3 \int_{\PP^1} (\underline {\bf \rho_{\tilde u/ab}}-\underline {\bf \rho_{\tilde v/ab}})\wedge [(dd^c \underline {\bf \rho_{\tilde u/ab}}+ \Omega)+(dd^c \underline {\bf \rho_{\tilde v/ab}}+ \Omega)].$$
\end{corollary}

\begin{proof} 
From (\ref{riheq})
$$\int_{\CC^2} (udd^cv-vdd^cu)\wedge dd^c(u+v)$$
$$=2\pi (ab)^3 \int_{\PP^1} (\underline {\bf \rho_{\tilde u/ab}}-\underline {\bf \rho_{\tilde v/ab}})\wedge [(dd^c \underline {\bf \rho_{\tilde u/ab}}+ \Omega)+(dd^c \underline {\bf \rho_{\tilde v/ab}}+ \Omega)].$$
We observe that
$$u(dd^cv)^2- v(dd^cu)^2= (udd^cv-vdd^cu)\wedge dd^c(u+v)+(v-u)dd^cu\wedge dd^cv.$$
Using this and the hypothesis $u\geq v$ gives the result.
\end{proof}

We also obtain a generalization of Theorem 6.9 of of \cite{BT}:

\begin{corollary} \label{bt6.9} Let $E,F$ be nonpluripolar compact subsets of $\CC^2$ with $E\subset F$. We have $\rho_{C,E}=\rho_{C,F}$ if and only if $V_{C,E}^*=V_{C,F}^*$ and $\hat E =\hat F \setminus P$ where $P$ is pluripolar.
\end{corollary}

\begin{proof} The ``if'' direction is obvious. For ``only if'' we may assume $E=\hat E$ and $F=\hat F$ since $V_{C,K}^*=V_{C,\hat K}^*$ for $K$ compact. It suffices to show $V_{C,E}^*\leq V_{C,F}^*$ as $E\subset F$ gives the reverse inequality. We have
$$0\leq  \int_{\CC^2}V_{C,E}^*(dd^cV_{C,F}^*)^2= \int_{\CC^2} [V_{C,E}^*(dd^cV_{C,F}^*)^2- V_{C,F}^*(dd^cV_{C,E}^*)^2]$$
since $V_{C,F}^*=0$ q.e. on $F$ (and hence a.e-$(dd^cV_{C,E}^*)^2$). Applying Corollary \ref{maincor} with $u=V_{C,E}^*$ and $v= V_{C,F}^*$, the right-hand-side of the displayed inequality is nonpositive since $\rho_{C,E}=\rho_{C,F}$ implies ${\bf \rho_{\tilde V_{C,E}^*/ab}}={\bf \rho_{\tilde V_{C,F}^*/ab}}$ on $\CC^2$ by (\ref{roheq}) so that $\underline {\bf \rho_{\tilde V_{C,E}^*/ab}}=\underline {\bf \rho_{\tilde V_{C,F}^*/ab}}$ on $\PP^1$. Hence $\int_{\CC^2}V_{C,E}^*(dd^cV_{C,F}^*)^2=0$. We conclude that $V_{C,E}^*=0$ a.e-$(dd^cV_{C,F}^*)^2$. By Proposition \ref{gctp}, $V_{C,E}^*\leq V_{C,F}^*$. Then $\hat E =\hat F \setminus P$ follows since $\{z\in \CC^2: V_{C,E}^*=0\}$ differs from $E=\hat E$ by a pluripolar set. 

\end{proof}


Again using Corollary \ref{maincor} and Proposition \ref{gctp}, we get an analogue of Lemma 2.1 in \cite{Bapp}, which is the key result for all that follows.

\begin{corollary} \label{keyresult} Let $K\subset \CC^2$ be compact and nonpluripolar and let $v\in L_C$ with $v\leq 0$ on $K$. Suppose that $\rho_v=\rho_{C,K}$ on $\partial P^2$. Then $v=V_{C,K}^*$ on $\CC^2\setminus \hat K$.
\end{corollary}

\begin{proof} Fix a constant $c$ so that $H_C(z) < c$ on $K$ and let 
$$w:=\max[v,0,H_C-c].$$
Then $w\in L_{C,+}$ with $w=0$ on $\hat K$ and since $H_C-c\leq V_{C,K}$ we have $\rho_w =\rho_{C,K}$ on $\partial P^2$. Then $\rho_w =\rho_{C,K}$ on $\CC^2$ (see Remark 4.4) and by (\ref{roheq}), ${\bf \rho_{\tilde w/ab}} ={\bf \rho_{\tilde V_{C,K}/ab}}$ on $\CC^2$. Thus $\underline {\bf \rho_{\tilde w/ab}} =\underline {\bf \rho_{\tilde V_{C,K}/ab}}$ on $\PP^1$. Since $w\leq V_{C,K}^*$, by Corollary \ref{maincor},
$$\int_{\CC^2} V_{C,K}^*(dd^cw)^2\leq \int_{\CC^2} w (dd^cV_{C,K}^*)^2=0,$$
the last equality due to $w=0$ on supp$(dd^cV_{C,K}^*)^2$. Thus $V_{C,K}^*=0$ a.e.-$(dd^cw)^2$ and hence $V_{C,K}^*\leq w$ a.e.-$(dd^cw)^2$. By Proposition \ref{gctp}, $V_{C,K}^*\leq w$ on all of $\CC^2$. Since $V_{C,K}^*\geq H_C-c$, $v=V_{C,K}^*$ on $\CC^2\setminus \hat K$.

\end{proof}

As in Theorem 2.1 in \cite{Bapp}, we get a sufficient condition for a sequence of polynomials to recover the $C-$extremal function of $K$ outside of $\hat K$. This will be used in section 8.

\begin{theorem} \label{thm64} Let $K\subset \CC^2$ be compact and nonpluripolar. Let $\{p_j\}$ be a sequence of polynomials, $p_j\in Poly(d_jC)$, with $\deg_C(p_j)=d_j$ such that
$$\limsup_{j\to \infty} ||p_j||_K^{1/d_j}=1 \ \hbox{and}$$
\begin{equation} \label{userob} \bigl(\limsup_{j\to \infty} \frac{1}{d_j}\log |\hat p_j(\zeta)|\bigr)^*=\rho_{C,K}(\zeta), \ \zeta \in \partial P^2.\end{equation}
Then
$$\bigl(\limsup_{j\to \infty} \frac{1}{d_j}\log |p_j(z)|\bigr)^*=V_{C,K}^*(z), \ z \in \CC^2 \setminus \hat K.$$
\end{theorem}

\begin{remark} Given an orthonormal basis $\{b_j\}$ of $\bigcup_n Poly(nC)$ in $L^2(\mu)$ where 
$\mu$ is a Bernstein-Markov measure for $K$, using 
$$\limsup_{j \to \infty} \tfrac{1}{\deg_C(b_j)}\log|b_j(z)|= V_{C,K}(z), \ z\not\in \hat K$$
from Corollary \ref{bmext}, in the proof of Proposition \ref{niv} we showed 
$$\limsup_{j \to \infty} \tfrac{1}{\deg_C(b_j)}\log|\hat b_j(\zeta)|= \rho_{C,K}(\zeta) \ \hbox{for} \ \zeta \in \partial P^2.$$
Theorem \ref{thm64} is a type of reverse implication.

\end{remark}

\begin{proof} The function
$$v(z):=\bigl(\limsup_{j\to \infty} \frac{1}{d_j}\log |p_j(z)|\bigr)^*$$
is psh in $\CC^2$. Given $\epsilon >0$, 
$$\frac{1}{d_j}\log |p_j(z)|\leq \epsilon, \ z\in K, \ j\geq j_0(\epsilon).$$
Thus 
$$\frac{1}{d_j}\log |p_j(z)|\leq V_{C,K}(z)+\epsilon, \ z\in \CC^2, \ j\geq j_0(\epsilon).$$
We conclude that $v\in L_C$ and $v\leq V_{C,K}$. Hence $\rho_v\leq \rho_{C,K}$. 

From Corollary \ref{keyresult}, to show $v= V_{C,K}^*$ outside $\hat K$ it suffices to show $\rho_v\geq \rho_{C,K}$ on $\partial P^2$. We use the argument from Proposition \ref{niv}. Recall from (\ref{uinabl}) for $u\in abL(\CC)$ we have
$$\limsup_{|t|\to \infty} [u(t)-ab\log |t|]=\inf_r\bigl(\max_{|t|=r}u(t)-ab\log r\bigr).$$ 
Fix $\zeta \in \partial P^2$ and apply the above to the function 
$$\lambda \to \frac{1}{d_j}\log |p_j(\lambda \circ \zeta)|=\frac{1}{d_j}\log |p_j(\lambda^a\zeta_1, \lambda^b \zeta_2)|.$$
We see that for any $r$
$$\frac{1}{d_j}\log |\hat p_j(\zeta)|=\limsup_{|\lambda|\to \infty} [ \frac{1}{d_j}\log |p_j(\lambda \circ \zeta)|-ab\log |\lambda|]$$
$$\leq \max_{|\lambda|=r} \frac{1}{d_j}\log |p_j(\lambda \circ \zeta)| -ab\log r.$$
Thus
$$\limsup_{j\to \infty} \frac{1}{d_j}\log |\hat p_j(\zeta)|\leq 
\limsup_{j\to \infty} \bigl(\max_{|\lambda|=r} \frac{1}{d_j}\log |p_j(\lambda \circ \zeta)| -ab\log r\bigr)$$
$$\leq \max_{|\lambda|=r} \bigl(\limsup_{j\to \infty}\frac{1}{d_j}\log |p_j(\lambda \circ \zeta)| -ab\log r\bigr)=\max_{|\lambda|=r} [v(\lambda \circ \zeta)-ab\log r]$$
where we used Hartogs lemma. Thus, letting $r\to \infty$,
$$\limsup_{j\to \infty} \frac{1}{d_j}\log |\hat p_j(\zeta)|\leq \rho_v(\zeta).$$
Since $\rho_v$ is usc, $\bigl(\limsup_{j\to \infty} \frac{1}{d_j}\log |\hat p_j(\zeta)|\bigr)^*\leq \rho_v(\zeta)$ and using the hypothesis (\ref{userob}) finishes the proof.

\end{proof}

\section{$C-$transfinite diameter and directional Chebyshev constants} From \cite{Sione}, we have a Zaharjuta-type proof of the existence of the limit 
\begin{equation} \label{tdlim}\delta_C(K):= \limsup_{n\to \infty}V_{n}^{1/l_n}\end{equation}
(see section 2) in the definition of $C-$transfinite diameter $\delta_C(K)$ of a compact set $K\subset \CC^d$ where $C$ satisfies (\ref{sigmainkp}). In the classical ($C=\Sigma$) case, Zaharjuta \cite{Zah} verified the existence of the limit in (\ref{tdlim}) by introducing directional Chebyshev constants $\tau(K,\theta)$ and proving 
$$\delta_{\Sigma}(K)=\exp \bigl(\frac{1}{|\sigma|}\int_{\sigma^0} \log \tau(K,\theta)dm(\theta)\bigr)$$
where $\sigma:=\{(x_1,...,x_d)\in \RR^d: 0\leq x_i \leq 1, \ \sum_{j=1}^d x_i = 1\}$ is the extreme ``face'' of $\Sigma$; $\sigma^0=\{(x_1,...,x_d)\in \RR^d: 0 < x_i <1, \ \sum_{j=1}^d x_i = 1\}$; and $|\sigma|$ is the $(d-1)-$dimensional measure of $\sigma$. 

In \cite{Sione}, a slight difference with the classical setting is that we have
\begin{equation}\label{zahtype} \delta_C(K)=\exp \bigl(\frac{1}{vol(C)}\int_C \log \tau(K,\theta)dm(\theta)\bigr) \end{equation}
where the directional Chebyshev constants $\tau(K,\theta)$ and the integration in the formula are over the entire $d-$dimensional convex body $C$. Moreover in the definition of $\tau(K,\theta)$ the standard grevlex ordering $\prec$ on $(\ZZ^+)^d$ (i.e., on the monomials in $\CC^d$) was used. This was required to obtain the submultiplicativity of the ``monic'' polynomial classes 
\begin{equation} \label{monicclass} M_k(\alpha):= \{p\in Poly(kC): p(z)=z^{\alpha} +\sum_{\beta \prec \alpha} c_{\beta}z^{\beta}\}\end{equation}
and corresponding Chebyshev constants
$$T_k(K,\alpha):=\inf \{||p||_K:p\in M_k(\alpha)\}^{1/k}.$$
However, in our triangle setting, following \cite{Sione} we can also define an ordering $\prec_C$ on $(\ZZ^+)^2$ by $\alpha =(\alpha_1,\alpha_2) \prec_C \beta= (\beta_1,\beta_2)$ if 
\begin{enumerate}
\item $\deg_C(z^{\alpha})< \deg_C(z^{\beta})$ or 
\item when $\deg_C(z^{\alpha})= \deg_C(z^{\beta})$ we have $\alpha_2 < \beta_2$.
\end{enumerate}
Then 
\begin{enumerate}
\item one has submultiplicativity of the corresponding ``monic'' polynomial classes defined as in (\ref{monicclass}) using $\prec_C$ (which we denote $M^{\prec_C}_k(\alpha)$) and one gets the formula (\ref{zahtype}) with $\theta \to \tau(K,\theta)$ continuous; and 
\item if $\phi=(\phi_1,\phi_2)$ is on the open hypotenuse $\mathcal C$ of $C$; i.e., $a\phi_1 +b\phi_2=ab$ with $\phi_1 \phi_2 >0$, and if $\phi = r\theta$ where $\theta$ lies on the interior of $C$ and $r>1$, then $r\log \tau(K,\theta)=\log \tau(K,\phi)$ (see \cite{SDRNA}, Lemma 5.4). (Note $\mathcal C=\sigma^0$ if $C=\Sigma\subset (\RR^+)^2$).
\end{enumerate}
As a consequence, in our triangle case $C=co\{(0,0),(b,0),(0,a)\}$,
\begin{equation}\label{zahtype2} \delta_C(K)=\exp \bigl(\frac{1}{\sqrt {a^2+b^2}}\int_{\mathcal C} \log \tau(K,\theta)d \theta \bigr) \end{equation}
where the directional Chebyshev constants $\tau(K,\theta)$ in (\ref{zahtype2}) and the integration in the formula are over $\mathcal C$ and 
$\theta \to \tau(K,\theta)$ is continuous on $\mathcal C$. In what follows, we fix our triangle $C$ and use this $\prec_C$ ordering to define these directional Chebyshev constants 
$$\tau(K,\theta):=\lim_{k\to \infty, \ \alpha/k\to \theta} T^C_k(K,\alpha)$$
for $\theta \in \mathcal C$ (in which case the limit exists) where
$$T^C_k(K,\alpha):=\inf \{||p||_K:p(z)=z^{\alpha} +\sum_{\beta \prec_C \alpha} c_{\beta}z^{\beta}\in M^{\prec_C}_k(\alpha)\}^{1/k}.$$

Using (\ref{zahtype2}) we have a result similar to Proposition 3.1 in \cite{Bfam}:

\begin{proposition} \label{prop71} For $E\subset F$ compact subsets of $\CC^2$, 
\begin{enumerate}
\item for all $\theta \in \mathcal C$, $\tau(E,\theta)\leq \tau(F,\theta)$ and 
\item $\delta_C(E)=\delta_C(F)$ if and only if $\tau(E,\theta)= \tau(F,\theta)$ for all $\theta \in \mathcal C$.
\end{enumerate}

\end{proposition}



\section{Polynomials approximating $V_{C,K}$} Following \cite{Bfam}, given a nonpluripolar compact set $K\subset \CC^2$ and $\theta \in \mathcal C$, a sequence of polynomials $\{Q_n\}$ is {\it $\theta-$asymptotically Chebyshev} (we write $\theta aT$) for $K$ if  
\begin{enumerate}
\item for each $n$ there exists $k_n \in \ZZ^+$ and $\alpha_n$ with $Q_n \in M^{\prec_C}_{k_n}(\alpha_n)$;
\item $\lim_{n\to \infty} k_n=+\infty$ and $\lim_{n\to \infty} \frac{\alpha_n}{k_n}=\theta$; and
\item $\lim_{n\to \infty} ||Q_n||_K^{1/k_n} =\tau (K,\theta).$

\end{enumerate}

\begin{proposition} \label{bloom3.2} Let $K\subset \CC^2$ be compact and nonpluripolar and satisfy $ e^{i\theta}\circ K =K$. Let $\{Q_n\}$ be $\theta aT$ for $K$.Then $\{\hat Q_n\}$ is $\theta aT$ for $K$.
\end{proposition}

\begin{proof} This follows from (\ref{hom}) giving $||\hat Q_n||_K \leq ||Q_n||_K$ for such $K$.
\end{proof}

Given $K\subset \CC^2$ compact and nonpluripolar, define
$$K_{\rho}:=\{z\in \CC^2: \rho_{C,K}(z)\leq 0\}.$$
Note that if $ e^{i\theta}\circ K =K$ then Theorem \ref{circlek} shows that $K=K_{\rho}$. Moreover, from Remark \ref{rem43},
$$\rho_{C,K}(e^{i\theta}\circ z)= \rho_{C,K}(z)$$
so that $ e^{i\theta}\circ K_{\rho} =K_{\rho}$ and $V_{C,K_{\rho}}^*=\rho^+_{C,K}:=\max[0,\rho_{C,K}]$.

\begin{theorem} \label{bloom3.1} Let $K\subset \CC^2$ be compact and regular and let $\{Q_n\}$ be $\theta aT$ for $K$ with $Q_n \in M^{\prec_C}_{k_n}(\alpha_n)$.Then $\{\hat Q_n\}$ is $\theta aT$ for $K_{\rho}$. Conversely, if $\{H_n\}$ is $\theta aT$ for $K_{\rho}$ with $H_n\in H_{k_n}(C)\cap M^{\prec_C}_{k_n}(\alpha_n)$, then the sequence $\{Tch_K H_n\}$ is $\theta aT$ for $K$. Moreover, 
$$\tau(K_{\rho},\theta)=\tau(K,\theta) \ \hbox{for all} \ \theta \in \mathcal C.$$

\end{theorem}

\begin{proof} Given $\{Q_n\}$ which are $\theta aT$ for $K$ with $Q_n \in M^{\prec_C}_{k_n}(\alpha_n)$, we have
$$\frac{1}{k_n} \log \frac{|Q_n(z)|}{||Q_n||_K}\leq V_{C,K}(z), \ z\in \CC^2.$$
Thus
$$\frac{1}{k_n} \log \frac{|\hat Q_n(z)|}{||Q_n||_K}\leq \rho_{C,K}(z), \ z\in \CC^2.$$
Hence for $z\in K_{\rho}$, 
$$\frac{1}{k_n} \log |\hat Q_n(z)|\leq \frac{1}{k_n} \log ||Q_n||_K; \ \hbox{i.e.}, \ ||\hat Q_n||_{K_{\rho}}\leq ||Q_n||_K.$$
Since
$$\lim_{n\to \infty} ||Q_n||_K^{1/k_n} =\tau (K,\theta),$$
\begin{equation} \label{tauk} \tau(K_{\rho},\theta)\leq \liminf_{n\to \infty} ||\hat Q_n||_{K_{\rho}}^{1/k_n} \leq \limsup_{n\to \infty} ||\hat Q_n||_{K_{\rho}}^{1/k_n}\leq \tau (K,\theta).\end{equation}

On the other hand, considering $\{H_n\}$ which are $\theta aT$ for $K_{\rho}$ (we can assume $H_n\in H_{k_n}(C)$ from Proposition \ref{bloom3.2}) we have 
$$\lim_{n\to \infty} ||H_n||_{K_{\rho}}^{1/k_n}  = \tau(K_{\rho},\theta).$$
Thus
$$ \limsup_{n\to \infty}\frac{1}{k_n} \log |H_n(z)|= \limsup_{n\to \infty}\frac{1}{k_n} \log \frac{|H_n(z)|}{||H_n||_{K_{\rho}}}+\log \tau(K_{\rho},\theta)$$
$$\leq \log \tau(K_{\rho},\theta) +\rho^+_{C,K}(z), \ z\in \CC^2.$$
By rescaling $K$; i.e., replacing $K$ by $rK$ for appropriate $r\geq 1$ if need be, we can assume that 
$$K_{\rho} \subset \{(z_1,z_2): |z_1|, |z_2|\leq 1\}.$$
In particular, $\rho_{C,K}\geq 0$ on $\partial P^2$. From Theorem \ref{sicbloom} we conclude that
$$ \limsup_{n\to \infty}\frac{1}{k_n} \log ||Tch_K H_n||_K \leq \log \tau(K_{\rho},\theta).$$
We have $Tch_K H_n\in M^{\prec_C}_{k_n}(\alpha_n)$ since $H_n\in H_{k_n}(C)\cap M^{\prec_C}_{k_n}(\alpha_n)$ and hence
\begin{equation} \label{tauk2}\tau (K,\theta)\leq \liminf_{n\to \infty} ||Tch_K H_n||_K^{1/k_n} \leq \limsup_{n\to \infty} ||Tch_K H_n||_K^{1/k_n}\leq \tau(K_{\rho},\theta).\end{equation}
Together with (\ref{tauk}) we conclude that $\tau(K_{\rho},\theta)=\tau(K,\theta)$ and the inequalities in (\ref{tauk}) and (\ref{tauk2}) are equalities.

\end{proof}

Finally, we utilize Theorems \ref{bloom3.1} and  \ref{thm64} together with Propositions \ref{tdv} and \ref{prop71} to prove our main result.

\begin{theorem} \label{mainthm} Let $K\subset \CC^2$ be compact and regular. Let $\{p_n\}$ be a countable family of polynomials with $p_n\in Poly(k_nC)$ such that for every $\theta \in \mathcal C$, there is a subsequence which is $\theta aT$ for $K$. Then
$$ \bigl( \limsup_{n\to \infty}\frac{1}{k_n} \log \frac{|p_n(z)|}{||p_n||_K}\bigr)^* = V_{C,K}(z), \ z\in \CC^2\setminus \hat K.$$

\end{theorem}

\begin{proof} Let
$$v(z):=\bigl( \limsup_{n\to \infty}\frac{1}{k_n} \log \frac{|p_n(z)|}{||p_n||_K}\bigr)^* .$$
Clearly $v\leq V_{C,K}$ on all of $\CC^2$. Let
$$w(z):=\bigl( \limsup_{n\to \infty}\frac{1}{k_n} \log \frac{|\hat p_n(z)|}{||p_n||_K}\bigr)^*.$$ 
To finish the proof, it suffices, by Theorem  \ref{thm64}, to show that 
$$w(z)=\rho_{C,K}(z), \ z \in \partial P^2.$$
Clearly $w\leq \rho_{C,K}$ in $\CC^2$ since $v\leq V_{C,K}$. To show the reverse inequality, we proceed as follows. Let
$$Z:=\{z\in \CC^2: w(z) <0\}.$$
Then $Z$ is open since $w$ is usc. We claim that $int(K_{\rho})\subset Z$. For if $z\in int(K_{\rho})$, we have 
$\rho_{C,K}(z) = -a<0$. Thus $w(z)\leq -a <0$ and $z\in Z$. Moreover, both sets $K_{\rho}$ and $Z$ satisfy the invariance property
$$ e^{i\theta}\circ K_{\rho} =K_{\rho} \ \hbox{and} \ e^{i\theta}\circ Z =Z$$
(for $Z$ this follows since $\hat p_n\in H_{k_n}(C)$). Thus to show the equality $w(z)=\rho_{C,K}(z)$ it suffices to verify the equality
$$int(K_{\rho}) =Z.$$
Suppose this is false. Then we take a point $z_0\in \partial K_{\rho}\cap Z$ and a closed ball $B$ centered at $z_0$ contained in $Z$. Since $B$ is regular and $K$ is assumed regular, by Proposition \ref{robinreg} together with Lemma 4.1 of \cite{Bfam}, $B\cup K_{\rho}$ is regular. 

Given $\theta \in \mathcal C$, by assumption there exists a subsequence $\NN_{\theta}\subset \NN$ such that $\{p_n\}_{n\in \NN_{\theta}}$ is $\theta aT$ for $K$. From Theorem \ref{bloom3.1}, $\{\hat p_n\}_{n\in \NN_{\theta}}$ is $\theta aT$ for $K_{\rho}$. Since $w\leq 0$ on $B\cup K_{\rho}$, for $z\in B\cup K_{\rho}$ we have
$$ \limsup_{n\in \NN_{\theta}}\frac{1}{k_n} \log |\hat p_n(z)|\leq \limsup_{n\in \NN_{\theta}}\frac{1}{k_n} \log ||p_n||_K=\log \tau (K,\theta).$$
Using Hartogs lemma, we conclude that
$$\log \tau(B\cup K_{\rho},\theta)\leq \limsup_{n\in \NN_{\theta}}\frac{1}{k_n} \log ||\hat p_n||_{B\cup K_{\rho}} \leq \log \tau (K,\theta).$$
Hence
$$\tau(B\cup K_{\rho},\theta) \leq \tau(K,\theta)=\tau(K_{\rho},\theta)$$
for all $\theta \in \mathcal C$. Since $\tau(B\cup K_{\rho},\theta) \geq \tau(K_{\rho},\theta)$ we see that
$$\tau(B\cup K_{\rho},\theta) = \tau(K_{\rho},\theta)$$
for all $\theta \in \mathcal C$. From Propositions \ref{tdv} and \ref{prop71} (and regularity of the sets $K_{\rho}, \ B\cup K_{\rho}$), 
$$V_{C,B\cup K_{\rho}}=V_{C,K_{\rho}}.$$
But $V_{C,K_{\rho}}=\rho_{C,K}^+=\max[0,\rho_{C,K}]$ thus if $B\setminus K_{\rho}\not = \emptyset$, since $\rho_{C,K}>0$ on $\CC^2\setminus K_{\rho}$ and $V_{C,B\cup K_{\rho}}=0$ on $B\cup K_{\rho}$, this is a contradiction.

\end{proof}

As examples of sequences of polynomials satisfying the hypotheses of Theorem \ref{mainthm}, as in \cite{Bfam} we have
\begin{enumerate}
\item the family $\{t_{k,\alpha}\in M^{\prec_C}_k(\alpha)\}_{k,\alpha}$ of Chebyshev polynomials (minimial supremum norm) for $K$ in these classes;
\item for a Bernstein-Markov measure $\mu$ on $K$, the corresponding polynomials $\{q_{k,\alpha}\in M^{\prec_C}_k(\alpha)\}_{k,\alpha}$ of minimal $L^2(\mu)$ norm (see Corollary \ref{bmext});
\item any sequence $p_{\alpha(s)}=z^{\alpha(s)}-L_{\alpha(s-1)}(z^{\alpha(s)})$ where $\{z^{\alpha(s)}\}$ is an enumeration of monomials with the $\prec_C$ order and $L_{\alpha(s-1)}(z^{\alpha(s)})$ is the Lagrange interpolating polynomial for the monomial $z^{\alpha(s)}$ at points $\{z_{s-1,j}\}_{j=1,...s-1}$ in the $(s-1)-$st row in a triangular array $\{z_{jk}\}_{j=1,2,...; \ k=1,...j}\subset K$ where the Lebesgue constants $\Lambda_{\alpha(s)}$ associated to the array grow subexponentially. Here, 
$$\Lambda_{\alpha(s)}:=\max_{z\in K} \sum_{j=1}^s |l_{sj}(z)|$$
where $l_{sj} \in Poly\bigl(\deg_C(z^{\alpha(s)})C\bigr)$ satisfies $l_{sj}(z_{sk})=\delta_{jk}, \ j,k=1,...,s$ and we require 
$$\lim_{s\to \infty} \Lambda_{\alpha(s)}^{1/\deg_C(z^{\alpha(s)})}=1.$$
We refer to Corollary 4.4 of \cite{Bfam} for details. 

\end{enumerate}

\begin{remark} Example (3) includes the case of a sequence of {\it $C-$Fekete polynomials} for $K$ (cf., p 1562 of \cite{Bfam}). The case of {\it $C-$Leja polynomials} for $K$, defined using $C-$Leja points as in \cite{Sione}, also satisfy the hypotheses of Theorem \ref{mainthm}.  This can be seen by following the proof of Corollary 4.5 in \cite{Bfam}. The proof that $C-$Leja polynomials satisfy the analogue of (4.28) in \cite{Bfam} is given in Theorem 1.1 of \cite{Sione}.

\end{remark}

\section{Further directions}

We reiterate that the arguments given in the note for triangles $C$ in $\RR^2$ with vertices $(0,0), (b,0), (0,a)$ where $a,b$ are relatively prime positive integers should generalize to the case of a simplex 
$$C=co\{(0,...,0),(a_1,0,...,0),...,(0,...,0,a_d)\}$$
in $(\RR^+)^d$ with $a_1,...,a_d$ pairwise relatively prime using the definition of the $C-$Robin function in Remark \ref{generald}. Indeed, following the arguments on pp. 72-82 of \cite{BloomLev} one should also be able to prove {\it weighted} versions of the $C-$Robin results for such simplices $C$
in $\RR^d$. We indicate the transition from the $C-$weighted situation for $d=2$ to a $\tilde C-$homogeneous unweighted situation for $d=3$. As in section 4, we lift the circle action on $\CC^2$, 
$$\lambda \circ (z_1,z_2):=(\lambda^az_1,\lambda^b z_2),$$
to $\CC^3$ via
$$\lambda \circ (t,z_1,z_2):=(\lambda t,\lambda^az_1,\lambda^b z_2).$$
Given a compact set $K\subset \CC^2$ and an admissible weight function $w\geq 0$ on $K$, i.e., $w$ is usc and $\{z\in K: w(z) >0\}$ is not pluripolar, we associate the set 
$$\tilde K_w:=\{(t\circ (1,z_1,z_2): (z_1,z_2)\in K, \ |t|=w(z_1,z_2)\}.$$
It follows readily that 
$$e^{i\theta} \circ \tilde K_w = \tilde K_w.$$
Setting $\tilde C=co\{(0,0,0),(1,0,0),(0,b,0),(0,0,a)\}$, we can relate a weighted $C-$Robin function $\rho^w_{C,K}$ to the $\tilde C-$Robin function $\rho_{\tilde C, \tilde K_w}$. Using these weighted ideas, the converse to Proposition \ref{tdv} should follow as in \cite{BloomLev}.

\end{document}